\documentclass[12pt]{article}
\usepackage{a4wide,amsmath,amssymb,amsthm,comment,amsrefs,mathdots,mathtools,amscd}
\usepackage[usenames]{color}
\usepackage[stable]{footmisc}

\newcommand{\abs}[1]{\left|{#1}\right|}
\newcommand{\norm}[1]{\lVert#1\rVert}
\newcommand{\GL}{\operatorname{GL}}
\newcommand{\Ad}{\operatorname{Ad}}
\newcommand{\swrz}{\mathcal{S}}
\newcommand{\Eisen}{\mathcal{E}}
\newcommand{\Kir}{\mathcal{K}}
\newcommand{\typKir}{KS}
\newcommand{\itint}{\int^{\operatorname{it}}}
\newcommand{\spn}{\operatorname{span}}
\newcommand{\AT}{(AT)}
\newcommand{\PD}{PD}
\newcommand{\lunq}{\le_{\operatorname{unq}}}
\newcommand{\mcp}{\kappa}
\newcommand{\tr}{\operatorname{tr}}
\newcommand{\mult}{\mathcal{M}}
\newcommand{\ext}{e}
\newcommand{\sprod}[2]{\left\langle{#1},{#2}\right\rangle}

\newcommand{\sqr}{\operatorname{sqr}}
\newcommand{\ord}{\operatorname{ord}}
\newcommand{\temp}{\operatorname{tmp}}
\newcommand{\Seg}{\mathcal{SEG}}
\newcommand{\Lie}{\operatorname{Lie}}
\newcommand{\Lieg}{\mathfrak{g}}
\newcommand{\one}{\mathbf{1}}
\newcommand{\data}{\mathfrak{d}}
\newcommand{\soc}{\operatorname{soc}}
\newcommand{\cs}{\mathfrak{c}}
\newcommand{\prt}{\mathcal{P}}
\newcommand{\pairing}{\Delta}
\renewcommand{\Re}{\operatorname{Re}}
\newcommand{\mc}{\operatorname{MC}}
\newcommand{\hmgns}{$m$-homogeneous}
\newcommand{\hmgnss}{m-\operatorname{hmgns}}
\newcommand{\expo}{\mathbf{e}}
\newcommand{\cusp}{\operatorname{cusp}}
\newcommand{\esqr}{\operatorname{esqr}}
\newcommand{\m}{\mathfrak{m}}
\newcommand{\Irr}{\operatorname{Irr}}
\newcommand{\gen}{\operatorname{gen}}
\newcommand{\Bil}{\mathcal{B}}
\newcommand{\trns}{\mathcal{T}}
\newcommand{\wi}{\hat w_i}
\newcommand{\wm}{\hat w_{m-1}}
\newcommand{\ww}{w_{m,n}}
\newcommand{\rest}{|}
\newcommand{\cmap}{\mathfrak{c}}
\newcommand{\Sh}{\operatorname{Sh}}
\newcommand{\typSh}{W_{\Sh}}
\newcommand{\Ze}{\operatorname{Ze}}
\newcommand{\typZe}{W_{\Ze}}
\newcommand{\A}{\mathbb{A}}
\newcommand{\Z}{\mathbb{Z}}
\newcommand{\C}{\mathbb{C}}
\newcommand{\R}{\mathbb{R}}
\newcommand{\fpsi}{\Psi}
\newcommand{\OOO}{\mathcal{O}}
\newcommand{\bs}{\backslash}
\newcommand{\rk}{\operatorname{rk}}
\newcommand{\diag}{\operatorname{diag}}
\newcommand{\Mat}{\operatorname{Mat}}
\newcommand{\Hom}{\operatorname{Hom}}
\newcommand{\Res}{\operatorname{Res}}
\newcommand{\Ind}{\operatorname{Ind}}
\newcommand{\ind}{\operatorname{ind}}
\newcommand{\Sp}{\operatorname{Sp}}
\newcommand{\mir}{D}
\newcommand{\Whit}{\mathcal{W}}

\newcommand{\sm}[4]{\left(\begin{smallmatrix}{#1}&{#2}\\{#3}&{#4}\end{smallmatrix}\right)}

\newenvironment{pmallmatrix}
{\left(\begin{smallmatrix}}
{\end{smallmatrix}\right)}

\newtheorem{theorem}{Theorem}[section]
\newtheorem{lemma}[theorem]{Lemma}

\newtheorem{proposition}[theorem]{Proposition}%[subsection]
\newtheorem{corollary}[theorem]{Corollary}%[subsection]
\newtheorem{question}[theorem]{Question}%[subsection]
\newtheorem{conjecture}[theorem]{Conjecture}%[subsection]
\theoremstyle{remark}
\newtheorem{remark}[theorem]{Remark}%[subsection]
\newtheorem{example}[theorem]{Example}%[subsection]

\begin{document}

\title{Local Rankin--Selberg integrals for Speh representations}
\author{Erez M. Lapid \and Zhengyu Mao}
\date{\today}
\maketitle

\begin{abstract}
We construct analogues of Rankin--Selberg integrals for Speh representations of the general linear group
over a $p$-adic field. The integrals are in terms of the Shalika model and
are expected to be the local counterparts of (suitably regularized) global integrals involving square-integrable
automorphic forms and Eisenstein series on the general linear group over a global field.
We relate the local integrals to the classical ones studied by Jacquet--Piatetski-Shapiro--Shalika.
We also introduce a unitary structure for Speh representation on the Shalika model, as well as various
other models including Zelevinsky's degenerate Whittaker model.
\end{abstract}

\setcounter{tocdepth}{1}
\tableofcontents

\section{Introduction}
The theory of Rankin--Selberg integrals for $\GL_n\times\GL_{n'}$, studied by Jacquet--Piatetski-Shapiro--Shalika
in a series of papers starting from the late 1970s, is a basic tool in the theory of automorphic forms
with an abundance of applications.
The theory is based on global zeta integrals (which involve Eisenstein series in the case $n'=n$)
that unfold to adelic integrals of Whittaker--Fourier coefficients of cuspidal representations.
By local multiplicity one, these integrals factorize into a product of local zeta integrals
which are pertaining to generic representations and their Whittaker models.

The purpose of this paper is to study a modification of the local Rankin-Selberg integrals
in the equal rank case for the so-called Speh representations.
The latter are the Langlands quotients $\Sp(\pi,m)$ of parabolic induction of essentially tempered representations
of the form $\pi\abs{\det}^{\frac{m-1}2}\otimes\pi\abs{\det}^{\frac{m-3}2}\otimes\dots\otimes\pi\abs{\det}^{\frac{m-1}2}$
where $\pi$ is a tempered representation of $\GL_n$ and $m\ge1$ is any integer.
(In the body of the paper we consider more generally generic $\pi$'s.)
%Here $m\ge1$ is an integer which is called the level.
These representations (of $\GL_{mn}$) are not generic if $m>1$ (i.e., they do not admit a Whittaker model).
Instead, the integrals involve the so-called Shalika model. (Strictly speaking, Shalika did not single out
Speh representations, and his model is more specific.)
It is known that any Speh representation admits a unique Shalika model, a fact which reflects the ``smallness''
of the Speh representation.
Structurally, the new integrals look very much like the classical ones and in fact they can be explicitly related.
In particular, the unramified computation reduces to that of the classical Rankin--Selberg integrals
(which in turn, uses Cauchy's identity and Shintani's formula for the unramified Whittaker function of $\GL_n$).

Just like the Whittaker model gives rise to the so-called Kirillov model (by restriction to the mirabolic subgroup,
namely, the stabilizer of a vector in $\GL_n$ in its standard $n$-dimensional representation)
the Shalika model gives rise to a closely related object which we call the Kirillov--Shalika model.
%Of course, in the case $m=1$ (where the Shalika model is just the Whittaker model) this amounts to the usual Kirillov model.
The role of the mirabolic subgroup is now played by the joint stabilizer of $m$ linearly independent vectors
in $\GL_{mn}$.
The argument of Gelfand--Kazhdan shows that the Kirillov--Shalika model contains
all functions that are compactly supported modulo the equivariance subgroup.

There are however some differences between the classical theory and its suggested analogue.
First, in the unramified case, we are unaware of a simple closed formula for the spherical function in the Shalika model,
except if $n\le2$ or if $n=3$ and $m=2$.
A related, equally difficult, problem is the asymptotic behavior of a function in the Shalika model.
Apart from the above-mentioned cases, both problems go beyond the ``comfort zone'' of spherical varieties.
Moreover, at this stage it is not clear whether there is analogue of the Bernstein--Zelevinsky theory of derivatives
in the case at hand. In particular, we do not know whether the restriction of $\Sp(\pi,m)$
to a parabolic subgroup of type $((n-1)m,m)$ is of finite length.

Another aspect of the paper is to provide an explicit, manifestly positive,
unitary structure for the Speh representation in its Shalika model.
Once again, this is modeled on the generic case in which Bernstein gives a unitary structure for
unitarizable representations on the Whittaker model by taking the $L^2$-inner product of Whittaker functions
restricted to the mirabolic subgroup.
For $m>1$ we use instead the joint stabilizer of $m$ vectors, as before.

Along with the abovementioned Shalika model, the Speh representation admits various other models, for instance
the degenerate Whittaker model considered by Zelevinsky (for any irreducible representation).
We can think of this as a sequence of models starting from the Zelevinsky model and ending with the Shalika model.
They all involve a character on a unipotent subgroup and are covered by the general construction of M\oe glin--Waldspurger.
The unipotent subgroups in the sequence are decreasing.
One can write down explicit isomorphisms (transition maps) between these models.
This idea was used by many authors, most recently and systematically by Gomez--Gourevitch--Sahi.
It also played a role in the recent work of Cai--Friedberg--Kaplan on new doubling constructions of $L$-functions.
We write an inner product for each of these models and show that the transition maps are unitary.

As far as we know, this is the first time a purely local, manifestly positive hermitian form for a general Speh representations
is given explicitly.
Of course, the intertwining operator on the standard module whose image is the Speh representation
induces its unitary structure -- a fact that is true in general for any unitarizable representation on a reductive group.
(In the case at hand we will explicitly relate this unitary structure to the one on the Shalika model.)
However the semidefiniteness of the intertwining operator is far from obvious -- in fact it is equivalent to unitarizability,
which is known to be a difficult problem in general.
Another realization of the inner product is obtained by using global theory to embed Speh representations
as local constituents of automorphic forms in the discrete spectrum of $\GL_{mn}$ over the adeles. Finally, in the $m=2$ case
one can also realize a Speh representation in the discrete spectrum of $L^2(H\bs\GL_{2n})$ where $H$ is the symplectic group
of rank $n$. However, there is no such analogue for $m>2$.

%\Erez{Hope to study dimension of fixed vectors under $K(n)$ using this}

In principle, the new local integrals are the local counterpart of certain global integrals, just as in the classical case.
However, in addition to Eisenstein series, these global integrals involve automorphic forms in the discrete spectrum,
rather than cusp forms, and they unfortunately do not converge (for any value of $s$).
It should be possible to carry out a regularization procedure to make sense of these integrals and to justify
the unfolding procedure. However, we will not discuss this aspect in the paper.

%\Erez{relation to Friedberg e al}

We now give some more details about the contents of the paper.
In \S\ref{sec: mhmgns}, we first introduce some notation and recall Zelevinsky's classification of irreducible representations
of the general linear group over a local non-archimedean field $F$.
We then introduce the class of \hmgns\ representations, which includes the Speh representations and which is the main focus of the paper.
In terms of Zelevinsky's classification they simply correspond to multisegments consisting
of segments of length $m$, where $m\ge1$ is a fixed integer parameter.
The case $m=1$ exactly corresponds to generic representations -- i.e., the classical theory.
In \S\ref{sec: models} we introduce the models pertaining to \hmgns\ representations, following M\oe glin--Waldspurger.
(In order to use their results, we assume from \S\ref{sec: models} onward that $F$ is of characteristic $0$,
although we expect that this assumption is superfluous.)
We also introduce the transition maps between the models.
They are given by integrals which entail no convergence issues.
Finally, we introduce the Kirillov--Shalika model which is the analogue of the classical Kirillov model for generic representations.
In \S\ref{sec: unit} we introduce a family of bilinear forms on a pair of models of \hmgns\ representations.
In the case where the two representations are in duality, these bilinear forms specialize to an invariant pairing,
at least in some cases. In the unitarizable case it gives rise to a manifestly positive invariant unitary structure.
The invariance is proved by induction on $m$ using Bernstein's theorem on invariant distributions with respect to the mirabolic subgroup.
In Appendix \ref{sec: apend} we relate this pairing to the one induced by the intertwining operator on the standard module.
In \S\ref{sec: locint} we define the local Rankin--Selberg integrals for \hmgns\ representations using their Shalika models.
Applying the transition maps we can express these integrals in terms of the Zelevinsky model.
Hence, we get their rationality in $q^s$, the unramified computation and some (but not complete) information about
their poles in general. We also relate these integrals to the abovementioned bilinear forms,
and in particular, to the invariant pairing.
In \S\ref{sec: n=m=2} we go back to the Kirillov--Shalika model and analyze in detail the
case of Speh representations of $\GL_4$ pertaining to supercuspidal representations of $\GL_2$.
We study the asymptotic behavior of a function in the Kirillov--Shalika model.
At this stage, it is hard to tell whether the result is representative of the general case or it is merely a low-rank fluke.
Finally, in \S\ref{sec: global} we write an informal global expression, modeled after the classical Rankin--Selberg integrals,
whose regularization is expected to unfold to the local integrals studied in the paper.
The regularization is necessary as the integral does not converge.
(It would also eliminate extraneous terms in the unfolding procedure.)
However, we do not discuss the regularization procedure and only give a purely heuristic argument.

\subsection*{Acknowledgement}
We thank Tobias Finis, Wee Teck Gan, Max Gurevich, Herv\'e Jacquet, Yiannis Sakellaridis and Akshay Venkatesh for useful discussions.
The paper was written while the first-named author was a member at the Institute for Advanced Study in Princeton.
He thanks the IAS warmly for its support and wonderful working conditions.

\section{\hmgns\ representations\footnote{This notion should not be confused with Zelevinsky's notion of homogenous
representations \cite{MR584084}}}
\label{sec: mhmgns}

\subsection{Notation}
Throughout the paper fix a non-archimedean local field $F$ with ring of integers $\OOO$ and absolute value $\abs{\cdot}$.
If $H$ is an algebraic group over $F$, we often also use $H$ to denote $H(F)$.

From section \ref{sec: models} onward $F$ is assumed to be of characteristic $0$
(although this assumption probably can be lifted).
In principle it should be possible to deal with the archimedean case as well with proper adjustments,
but we do not consider this case here.

We will consider complex, smooth representations of finite length of the groups $\GL_n(F)$, $n\ge0$.
We denote the set of irreducible representations of $\GL_n(F)$ (up to equivalence) by $\Irr\GL_n$
and set $\Irr=\cup_{n\ge0}\Irr\GL_n$. We write $\Irr\GL_0=\{\one\}$.
(In contrast, the one-dimensional trivial character of $\GL_1$ will be denoted by $\one_{F^*}$.)
The subset of supercuspidal (resp., square-integrable, essentially square-integrable, tempered, generic) representations will be denoted by
$\Irr_{\cusp}$ (resp., $\Irr_{\sqr}$, $\Irr_{\esqr}$, $\Irr_{\temp}$, $\Irr_{\gen}$).

Let $\pi$ be a representation of $\GL_n(F)$.
We denote by $\pi^\vee$ the contragredient of $\pi$
and by $\soc(\pi)$ the socle of $\pi$ (the maximal semisimple subrepresentation of $\pi$).
If $\pi$ is non-zero then we write $\deg\pi=n$, the degree of $\pi$.
For any character $\omega$ of $F^*$ (i.e., $\omega\in\Irr\GL_1$)
we denote by $\pi\omega$ the representation obtained from $\pi$ by twisting by the character $\omega\circ\det$.
For instance, $\pi\abs{\cdot}$ is the twist of $\pi$ by $\abs{\det}$.
We also write $J_P(\pi)$ for the (normalized) Jacquet module of $\pi$ with respect to a parabolic subgroup
$P$ of $\GL_n$, defined over $F$.
If $\tau\in\Irr\GL_n$ then we write $\tau\le\pi$ if $\tau$ occurs as a subquotient of $\pi$.
If in addition $\tau$ occurs with multiplicity one in the Jordan--H\"older sequence of $\pi$
then we write $\tau\lunq\pi$.

If $\pi_1,\dots,\pi_k$ are representations of $\GL_{n_1}(F),\dots,\GL_{n_k}(F)$ respectively
then we denote the representation parabolically induced
from $\pi_1\otimes\dots\otimes\pi_k$ (normalized induction), with respect to the standard
parabolic subgroup of block upper triangular matrices, by $\pi_1\times\dots\times\pi_k$
and refer to it as the product representation.
We also use the notation $\Ind_H^G$ and $\ind_H^G$ to denote induction and compact induction
(both normalized) from a subgroup $H$ of $G$.

For any $\tau\in\Irr_{\esqr}$ let $\expo(\tau)$ be the unique real number $s$ such that the
twisted representation $\tau\abs{\cdot}^{-s}$ is unitarizable (i.e., has a unitary central character).
Note that $\expo(\tau^\vee)=-\expo(\tau)$.
Any $\pi\in\Irr_{\gen}$ can be written uniquely (up to permutation) as $\pi=\tau_1\times\dots\times\tau_k$ where
$\tau_i$ are essentially square-integrable. %Moreover, up to permutation $\tau_1,\dots,\tau_k$ are uniquely determined by $\pi$.
Let $\expo(\pi)=\min\expo(\tau_i)$.
Then $\expo(\pi)+\expo(\pi^\vee)\le0$ with equality if and only if $\pi$ is essentially tempered
and $\pi$ is tempered if and only if $\expo(\pi)=\expo(\pi^\vee)=0$.
More generally, we will say that $\pi$ is ``approximately tempered'' \AT\ if $\expo(\pi)+\expo(\pi^\vee)+1>0$.
Equivalently, $\expo(\tau_i)-\expo(\tau_j)<1$ for all $i,j$.
It is known that every unitarizable $\pi\in\Irr_{\gen}$ is \AT.
We denote by $\Irr_{\AT}$ the set of \AT\ representations.

For any set $A$ we denote by $\mult(A)$ the free commutative monoid generated by $A$,
considered as an ordered monoid.
Thus, an element of $\mult(A)$ (a multiset of $A$) is a finite (possibly empty) formal sum of element of $A$.

\subsection{Zelevinsky classification}
We recall the well-known results and terminology of \cite{MR584084}.

A \emph{segment} $\Delta$ (of length $l>0$ and center $\rho\in\Irr_{\cusp}$)
is a non-empty finite subset of $\Irr_{\cusp}$ of the form
\[
\Delta_\rho^{(l)}=\{\rho\abs{\cdot}^{\frac{1-l}2},\rho\abs{\cdot}^{\frac{3-l}2},\dots,\rho\abs{\cdot}^{\frac{l-1}2}\}.
\]
We define $\deg\Delta=l\deg\rho$ and write $\ext(\Delta)=\rho\abs{\cdot}^{\frac{l-1}2}\in\Irr_{\cusp}$
(the endpoint of $\Delta$),
\[
\cs(\Delta)=\rho\abs{\cdot}^{\frac{1-l}2}+\rho\abs{\cdot}^{\frac{3-l}2}+\dots+\rho\abs{\cdot}^{\frac{l-1}2}\in\mult(\Irr_{\cusp})
\]
and $\Delta^\vee=\Delta_{\rho^\vee}^{(l)}$.
For compatibility we also write $\Delta_\rho^{(0)}=\emptyset$.
Denote by $\Seg$ the set of all segments.
We extend $\deg$, $\ext$ and $\cs$ to $\mult(\Seg)$ additively.

For any $\Delta=\Delta_\rho^{(l)}\in\Seg$ let
\[
Z(\Delta)=\soc(\rho\abs{\cdot}^{\frac{1-l}2}\times\rho\abs{\cdot}^{\frac{3-l}2}\times\dots\times\rho\abs{\cdot}^{\frac{l-1}2})
\in\Irr\GL_{\deg\Delta}.
\]
(For compatibility we also set $Z(\emptyset)=\one$.)
Then $Z(\Delta)^\vee=Z(\Delta^\vee)$.
Given $\Delta_1,\Delta_2\in\Seg$ we write $\Delta_2\prec\Delta_1$ if
$\Delta_i=\Delta_{\rho_i}^{(l_i)}$ with $\rho_2\abs{\cdot}^{\frac{1-l_2}2+\alpha}=\rho_1\abs{\cdot}^{\frac{1-l_1}2}$
for some $\alpha\in\Z_{>0}$ such that $l_2-l_1<\alpha\le l_2$.
If either $\Delta_2\prec\Delta_1$ or $\Delta_1\prec\Delta_2$ then we say that $\Delta_1$ and $\Delta_2$ are \emph{linked}.
The induced representation $Z(\Delta_1)\times Z(\Delta_2)$ is reducible if and only if $\Delta_1$ and $\Delta_2$ are linked.

The well-known classification result of Zelevinsky \cite[Theorem 6.5]{MR584084} extends the map $\Delta\mapsto Z(\Delta)$
to a degree preserving bijection
\[
\m\mapsto Z(\m)
\]
between $\mult(\Seg)$ and $\Irr$. If $\m=\Delta_1+\dots+\Delta_k$ and $\Delta_i\not\prec\Delta_j$ for any $i<j$
(which can always be arranged) then $Z(\m)=\soc(Z(\Delta_1)\times\dots\times Z(\Delta_k))$.
An element of $\mult(\Seg)$ is called a multisegment.
We have $Z(\m)^\vee=Z(\m^\vee)$ where we extend $^\vee$ from $\Seg$ to $\mult(\Seg)$ additively.
For any $\m_1,\m_2\in\mult(\Seg)$ we have $Z(\m_1+\m_2)\le Z(\m_1)\times Z(\m_2)$.
In particular, if $Z(\m_1)\times Z(\m_2)$ is irreducible then it is equal to $Z(\m_1+\m_2)$.
\begin{multline} \label{eq: suffcond}
\text{A sufficient condition for the irreducibility of $Z(\m_1)\times Z(\m_2)$ is that every segment}\\
\text{which occurs in $\m_1$ is unlinked with every segment which occurs in $\m_2$.}
\end{multline}
By identifying an irreducible supercuspidal representation with a singleton segment we view
$\mult(\Irr_{\cusp})$ as a submonoid of $\mult(\Seg)$. The map $Z$ restricts to a bijection
\[
\mult(\Irr_{\cusp})\rightarrow\Irr_{\gen}.
\]
An element of $\mult(\Irr_{\cusp})$ is a called a cuspidal data.
We write $\cs(Z(\m))=\cs(\m)$. The resulting map
\[
\cs:\Irr\rightarrow\mult(\Irr_{\cusp})
\]
is the supercuspidal support (which of course can be defined without reference to the Zelevinsky classification).

For any segment $\Delta=\Delta_\rho^{(l)}$ let $\Delta^-=\Delta_{\rho\abs{\cdot}^{-\frac12}}^{(l-1)}$
denote either the segment obtained by removing the endpoint $\ext(\Delta)$ of $\Delta$ if $l>1$
or the empty set otherwise.

Let now $\sigma=Z(\m)$ where $\m=\Delta_1+\dots+\Delta_k$. Let
\[
\m^-=\Delta_1^-+\dots+\Delta_k^-\text{ (disregarding empty sets)}.
\]
Define recursively, $\m^{(0)}=\m$ and $\m^{(k)}=(\m^{(k-1)})^-$, $k>0$ with $\m^{(l)}=0$, $l$ minimal.
Let $n_k=\deg\ext(\m^{(k-1)})$, $k=1,\dots,l$ so that $n_1+\dots+n_l=\deg\sigma$
and let $\omega_k=Z(\ext(\m^{(k-1)}))\in\Irr_{\gen}\GL_{n_k}$.
Let $P=P_\sigma=M_\sigma\ltimes U_\sigma=M\ltimes U$ be the standard parabolic subgroup of type $(n_l,\dots,n_1)$. \label{sec: Psig}
By \cite[\S8.3]{MR584084} the Jordan--H\"older sequence of $J_P(\sigma)$ admits a unique generic irreducible representation $\omega$ of $M$
and moreover $\omega\lunq J_P(\sigma)$.
Equivalently, $\dim\Hom_N(\sigma,\psi_P)=\dim\Hom(\sigma,\Ind_N^G\psi_P)=1$ where
$N$ is the maximal nilpotent group of upper unitriangular matrices,
$\psi_P$ is a character of $N$ which is trivial on $U$ and non-degenerate on $N_M=N\cap M$.
(This property determines $P$ uniquely up to association.)
Moreover, $\omega=\omega_l\otimes\dots\otimes\omega_1$ (see e.g., \cite[Lemma 9.17]{MR3178433}).
(For an arbitrary $P$, $\Hom_N(\sigma,\psi_P)$ is finite-dimensional.)
We will call the image of $\sigma$ in $\Ind_N^G\psi_P$ the \emph{Zelevinsky model} of $\sigma$.
In general, it is not true that $\sigma$ is a subrepresentation of $\omega_l\times\dots\times\omega_1$.
For example if $\sigma=Z(\{\one_{F^*},\abs{\cdot}\}+\{\one_{F^*}\})$
then $l=2$, $\omega_2=\one_{F^*}$, $\omega_1=Z(\{\abs{\cdot}\}+\{\one_{F^*}\})$ and $\omega_2\times\omega_1$ is irreducible
(and generic).

\subsection{Ladder representations}
A multisegment $\m$ is called a (strict) ladder if it can be written as $\m=\Delta_1+\dots+\Delta_k$
where $\Delta_{i+1}\prec\Delta_i$ for all $i=1,\dots,k-1$.

\begin{lemma} \label{lem: irredlad}
\begin{enumerate}
\item \cite[Lemma 6.17]{MR3573961}
Let $\pi_1,\dots,\pi_k$ be ladder representations. Then $\pi_1\times\dots\times\pi_k$ is irreducible
if and only if $\pi_i\times\pi_j$ is irreducible for all $i,j$.
\item \cite[Lemma 6.21]{MR3573961} Suppose that $Z(\m_1)$ and $Z(\m_2)$ are two ladder representations.
Suppose that each segment of $\m_1$ occurs in $\m_2$. Then $Z(\m_1)\times Z(\m_2)$ is irreducible.
\end{enumerate}
\end{lemma}

The Jacquet module of ladder representations was described in \cite{MR2996769}.
The following is an immediate consequence.
\begin{lemma}\cite{MR2996769} \label{lem: Jmlad}
Let $\pi=Z(\m)$ be a ladder representation and $P$ a maximal parabolic subgroup.
Then
\begin{enumerate}
\item $J_P(\pi)$ is a direct sum of irreducible representations of the form $\tau\otimes\omega$
where both $\tau$ and $\omega$ are products of ladder representations.
\item \label{part: nongen}
If $\tau\otimes\omega\le J_P(\pi)$ and $\omega\notin\Irr_{\gen}$ then there exists $\rho\in\Irr_{\cusp}$
such that $\rho\le\cs(\omega)$, $\rho\abs{\cdot}\le\ext(\m)$ but $\rho\not\le\ext(\m)$.
\item \label{part: gen} If $\tau\otimes\omega\le J_P(\pi)$ with $\omega\in\Irr_{\gen}$ then $\cs(\omega)\le\ext(\m)$.
Moreover, if $\rho\in\Irr_{\cusp}$ is such that $\rho\le\ext(\m)$ and $\rho\abs{\cdot}\le\cs(\omega)$ then $\rho\le\cs(\omega)$.
\item \label{part: atmostonce} If $\tau\otimes\omega\le J_P(\pi)$ and $\rho\in\Irr_{\cusp}$ is such that $\rho\abs{\cdot}\not\le\cs(\omega)$
then $\rho$ occurs in $\cs(\omega)$ with multiplicity at most one.
\end{enumerate}
\end{lemma}

\subsection{}

From now on let $m,n\ge1$ be integers and $G=\GL_{mn}$.
For simplicity we write $\Delta_\rho=\Delta_\rho^{(m)}$.

We say that $\sigma\in\Irr G$ is \emph{\hmgns}\ if $\sigma=Z(\Delta_1+\dots+\Delta_k)$ where
each $\Delta_i$ is of length $m$.
(If $m=1$ this simply means that $\sigma$ is generic.)
For any $\pi=Z(\{\rho_1\}+\dots+\{\rho_k\})\in\Irr_{\gen}$ define
\[
\Sp(\pi,m)=Z(\Delta_{\rho_1}+\dots+\Delta_{\rho_k})\in\Irr.
\]
It is clear that the map $\pi\mapsto\Sp(\pi,m)$ defines a bijection between $\Irr_{\gen}\GL_n$ and
the set $\Irr_{\hmgnss}G$ of irreducible \hmgns\ representations of $G$.
We have $\Sp(\pi,m)^\vee=\Sp(\pi^\vee,m)$.

\begin{remark}
This notion is very close to the concept of ``representations of type $(n,m)$'' introduced in \cite{1802.02637}
except that we only consider irreducible representations.
\end{remark}

\begin{remark} \label{rem: unr}
In general, if $\pi$ is unramified (and generic) then $\Sp(\pi,m)$ is not necessarily unramified if $m>1$.
More precisely, if $\pi=Z(\{\rho_1\}+\dots+\{\rho_k\})$ is unramified (so that $\rho_i$ are unramified characters
of $F^*$ and $\rho_i\ne\rho_j\abs{\cdot}$ for all $i,j$) then $\Sp(\pi,m)$ is unramified if and only if
$\Delta_{\rho_1}^{(m)},\dots,\Delta_{\rho_k}^{(m)}$ are mutually unlinked. For instance, this is the case if $\pi$ is \AT.
\end{remark}

Suppose that $\sigma=\Sp(\pi,m)$ with $\pi\in\Irr_{\gen}\GL_n$. Then in the notation of \S\ref{sec: Psig}
$P_\sigma=P_{m,n}=M\ltimes U$ is the standard parabolic subgroup of $G$ of type $(\overbrace{n,\dots,n}^m)$,
consisting of the block upper triangular matrices with blocks of size $n\times n$.
Thus, $M\simeq\overbrace{\GL_n\times\dots\times\GL_n}^m$.

Just as in the case $m=1$, there are simple building blocks for \hmgns\ representations.

\begin{proposition}
Let $\sigma=\Sp(\pi,m)\in\Irr G$ be \hmgns. Then there exist $\pi_1,\dots,\pi_t\in\Irr_{\gen}$ such that
\begin{enumerate}
\item \label{part: sigmalad} $\sigma_i:=\Sp(\pi_i,m)$ is a ladder representation for all $i$.
\item $\Sp(\pi,l)=\Sp(\pi_1,l)\times\dots\times\Sp(\pi_t,l)$ for all $l=1,\dots,m$.
In particular, $\pi=\pi_1\times\dots\times\pi_t$ and $\sigma=\sigma_1\times\dots\times\sigma_t$.
\end{enumerate}
%%%The property \ref{part: sigmalad} and \eqref{eq: sigmaprod} determine $\pi_1,\dots,\pi_t\in\Irr_{\gen}$ uniquely.
Moreover, let $Q$ be the maximal standard parabolic subgroup of type $((m-1)n,n)$
and denote by $J_Q(\sigma)_{;\data}$ the direct summand of $J_Q(\sigma)$ pertaining to the supercuspidal data
$\data\in\mult(\Irr_{\cusp})$ in the second ($\GL_n$) factor. Then
\[
J_Q(\sigma)_{;\cs(\pi)\abs{\cdot}^{\frac{m-1}2}}=\Sp(\pi,m-1)\abs{\cdot}^{-\frac12}\otimes\pi\abs{\cdot}^{\frac{m-1}2}.
\]
\end{proposition}

\begin{proof}
Write $\pi=Z(\sum_{i\in I}\{\rho_i\})$ with $\rho_i\in\Irr_{\cusp}$ and let $l\ge1$.
We say that a subset $J$ of $I$ is an $l$-chain if it can be written (uniquely) as
$J=\{i_1,\dots,i_r\}$ where for all $j=1,\dots,r-1$
we have $\rho_{i_j}=\rho_{i_{j+1}}\abs{\cdot}^{\alpha_j}$ with $\alpha_j\in\{1,\dots,l\}$.
(For example, a $1$-chain is simply a segment.)
Clearly, $J$ is an $l$-chain if and only if $Z(\sum_{j\in J}\Delta_{\rho_j}^{(l)})$ is a ladder representation.

We say that two partitions of $I$ are equivalent if one can be obtained from the other by applying a permutation $\tau$ of $I$
such that $\rho_{\tau(i)}=\rho_i$ for all $i$.
It is easy to see that for any $l\ge 1$ there exists a partition $\prt^{(l)}(I)$ of $I$
consisting of $l$-chains, such that for any $J,J'\in\prt^{(l)}(I)$ at least one of the following conditions holds.
\begin{enumerate}
\item $\{\rho_j:j\in J\}\subset\{\rho_j:j\in J'\}$.
\item $\{\rho_j:j\in J'\}\subset\{\rho_j:j\in J\}$.
\item For every $j\in J$ and $j'\in J'$ the segments $\Delta_{\rho_j}^{(l)}$ and $\Delta_{\rho_{j'}}^{(l)}$ are unlinked.
\end{enumerate}
Moreover, $\prt^{(l)}(I)$ is unique up to equivalence.
Indeed, $\prt^{(l)}(I)$ can be defined inductively by taking a maximal $l$-chain $J$ of $I$ (with respect to inclusion)
together with the partition $\prt^{(l)}(I\setminus J)$.
It follows from this description that up to equivalence, $\prt^{(l_1)}(I)$ is a refinement of $\prt^{(l_2)}(I)$ if $l_2\ge l_1$.

For any $J\subset I$ let $\pi_J=Z(\sum_{j\in J}\{\rho_j\})\in\Irr_{\gen}$ and $\sigma_J=\Sp(\pi_J,m)$.
Then $\sigma_J$ is a (\hmgns) ladder representation for any $J\in\prt^{(m)}(I)$.
It follows from the defining property of $\prt^{(m)}(I)$, \eqref{eq: suffcond} and Lemma \ref{lem: irredlad}
that $\bigtimes_{J\in\prt^{(m)}(I)}\sigma_J$ is irreducible, hence equals $\sigma$.

Likewise, for any $l\le m$ we have $\Sp(\pi,l)=\bigtimes_{J'\in\prt^{(l)}(I)}\Sp(\pi_{J'},l)$.
Since $\Sp(\pi_J,l)=\bigtimes_{J'\in\prt^{(l)}(I):J'\subset J}\Sp(\pi_{J'},l)$ for all $J\in\prt^{(m)}(I)$
we infer that $\Sp(\pi,l)=\bigtimes_{J\in\prt^{(m)}(I)}\Sp(\pi_J,l)$.
In particular, $\pi=\times_{J\in\prt^{(m)}(I)}\pi_J$.

By \cite{MR2996769} we have
\[
\Sp(\pi_J,m-1)\abs{\cdot}^{-\frac12}\otimes\pi_J\abs{\cdot}^{\frac{m-1}2}\lunq J_{((m-1)n_J,n_J)}(\sigma_J)
\]
for all $J\in\prt^{(m)}(I)$ where $n_J=\deg\pi_J$. Therefore,
\begin{multline*}
\Sp(\pi,m-1)\abs{\cdot}^{-\frac12}\otimes\pi\abs{\cdot}^{\frac{m-1}2}=
\\\times_{J\in\prt^{(m)}(I)}\Sp(\pi_J,m-1)\abs{\cdot}^{-\frac12}\otimes
\times_{J\in\prt^{(m)}(I)}\pi_J\abs{\cdot}^{\frac{m-1}2}\le J_Q(\sigma).
\end{multline*}
On the other hand, suppose that $\tau_J,\omega_J\in\Irr$ with $\tau_J\otimes\omega_J\le J_P(\sigma_J)$, $J\in\prt^{(m)}(I)$ and
\begin{equation} \label{eq: suppeq}
\sum_{J\in\prt^{(m)}(I)}\cs(\omega_J)=\cs(\pi)\abs{\cdot}^{\frac{m-1}2}.
\end{equation}
We claim that this is possible only if
$\tau_J=\Sp(\pi_J,m-1)\abs{\cdot}^{-\frac12}$ and $\omega_J=\pi_J\abs{\cdot}^{\frac{m-1}2}$
for all $J$. We prove it by induction on $\deg\pi$. The base of the induction is trivial.
For the induction step, it is enough to prove that if $J$ is a maximal $m$-chain then $\omega_J=\pi_J\abs{\cdot}^{\frac{m-1}2}$.
We use Lemma \ref{lem: Jmlad}.
By part \ref{part: nongen}, if $J$ is a maximal $m$-chain then $\omega_J$ is generic.
For otherwise, there would exist $i\in I$ such that $\rho_i\notin\{\rho_j:j\in J\}$
but $\rho_i\abs{\cdot}\in\{\rho_j:j\in J\}$ in contradiction to the maximality of $J$.
On the other hand, by part \ref{part: atmostonce} if $\rho\abs{\cdot}\not\le\cs(\pi)\abs{\cdot}^{\frac{m-1}2}$
then $\rho$ can occur in $\cs(\omega_J)$ at most once for any $J\in\prt^{(m)}(I)$.
It follows from \eqref{eq: suppeq} that if $\rho\abs{\cdot}\not\le\cs(\pi)\abs{\cdot}^{\frac{m-1}2}$ then
$\rho\le\cs(\omega_J)$ if and only if  $\rho=\rho_j\abs{\cdot}^{\frac{m-1}2}$ for some $j\in J$.
By part \ref{part: gen} it now follow that if $J$ is a maximal $m$-chain then
$\cs(\omega_J)=\sum_{j\in J}\{\rho_j\abs{\cdot}^{\frac{m-1}2}\}$ as required.

This finishes the proof.
\end{proof}

\begin{remark}
It can be shown that up to permutation, $\sigma_1,\dots,\sigma_t$ are the unique ladder representations
such that $\sigma=\sigma_1\times\dots\times\sigma_t$. We will not need to use this fact.
\end{remark}

By Frobenius reciprocity and \cite[Corollary 4.10]{MR3573961} we conclude
\begin{corollary} \label{cor: indSp}
For any $\pi\in\Irr_{\gen}\GL_n$,
\[
\Sp(\pi,m)=\soc(\Sp(\pi,m-1)\abs{\cdot}^{-\frac12}\times\pi\abs{\cdot}^{\frac{m-1}2})\lunq
\Sp(\pi,m-1)\abs{\cdot}^{-\frac12}\times\pi\abs{\cdot}^{\frac{m-1}2}.
\]
\end{corollary}

By induction on $m$ we get
\begin{corollary} \label{cor: substd}
For any $\pi\in\Irr_{\gen}\GL_n$, $\Sp(\pi,m)$ is a subrepresentation of
\[
\Pi:=\pi\abs{\cdot}^{\frac{1-m}2}\times\pi\abs{\cdot}^{\frac{3-m}2}\times\dots\times\pi\abs{\cdot}^{\frac{m-1}2}.
\]
Equivalently (by passing to the contragredient), $\Sp(\pi,m)$ is a quotient of
\[
\tilde\Pi:=\pi\abs{\cdot}^{\frac{m-1}2}\times\pi\abs{\cdot}^{\frac{m-3}2}\times\dots\times\pi\abs{\cdot}^{\frac{1-m}2}.
\]
\end{corollary}

\begin{remark}
If $\pi$ is \AT\ then $\Sp(\pi,m)$ (known as ``Speh representation'') is the Langlands quotient of $\tilde\Pi$.
In particular, $\Sp(\pi,m)$ is the image of the standard intertwining operator from $\tilde\Pi$ to $\Pi$.
Moreover, $\Sp(\pi,m)=\soc(\Pi)\lunq\Pi$.
However, in general for $m>2$ and $\pi\in\Irr\GL_n$ it is not true that $\soc(\Pi)\lunq\Pi$.
For instance, if $\pi=\abs{\cdot}\times\abs{\cdot}^{-1}\in\Irr_{\gen}\GL_2$ then $\Sp(\pi,3)$
occurs with multiplicity two in the Jordan--H\"older sequence of
$\pi\abs{\cdot}^{-1}\times\pi\times\pi\abs{\cdot}$.
Note that in this case we still have $\Sp(\pi,m)=\soc(\Pi)$.
We do not know whether this holds in general, i.e., whether $\soc(\Pi)$ is irreducible.
\end{remark}

\section{The models} \label{sec: models}
\subsection{}
Throughout this section, fix $\pi\in\Irr_{\gen}\GL_n$ and let $\sigma=\Sp(\pi,m)\in\Irr G$ and
$P=P_\sigma=P_{m,n}=M\ltimes U$ the standard parabolic subgroup of $G$ of type $\overbrace{(n,\dots,n)}^m$.
Let $\bar U=\,^tU$ be the opposite of $U$.
Fix a non-trivial character $\psi$ of $F$. Let $\fpsi$ be the function on $G$ given by
\[
\fpsi(g)=\psi(\sum_{1\le i<nm:n\nmid i}g_{i,i+1}).
\]
We denote the restriction of $\fpsi$ to a subset $A$ of $G$ by $\fpsi_A$.
Let $N=N_{m,n}$ (resp., $\bar N=\,^tN$) be the group of upper (resp., lower) unitriangular matrices in $G$.
Then $\fpsi_N$ is a (degenerate) character on $N$ which is trivial on $U$ and non-degenerate on $N_M:=N\cap M$.
Recall that $\Hom_G(\sigma,\Ind_N^G\fpsi_N)$ is one-dimensional.

Denote by $\Whit^{\fpsi_N}(\sigma)$ the image of $\sigma$ in $\Ind_N^G\fpsi_N$, i.e.~the Zelevinsky model of $\sigma$.
By Corollaries \ref{cor: indSp} and \ref{cor: substd}, for any $\typZe\in\Whit^{\fpsi_N}(\sigma)$ we have
\begin{subequations}
\begin{equation} \label{eq: maxrest}
(\abs{\det}^{\frac{1-n}2}\otimes\abs{\det}^{(m-1)(n-1)/2})\typZe\rest_{\GL_{(m-1)n}\times\GL_n}
\in\Whit^{\fpsi_{N\cap(\GL_{(m-1)n}\times\GL_n)}}(\Sp(\pi,m-1)\otimes\pi),
\end{equation}
\begin{equation} \label{eq: restW}
\delta_P^{-\frac12}\delta'\typZe\rest_M\in\Whit^{\fpsi_{N_M}}(\pi^{\otimes m}),
\end{equation}
\end{subequations}
where $\pi^{\otimes m}=\overbrace{\pi\otimes\dots\otimes\pi}^m$, $\delta_P$ is the modulus character of $P$
and $\delta'=\delta_P^{\frac1{2n}}$ is the character of $M$ given by
\[
\delta'(\diag(g_1,\dots,g_m))=\abs{\det g_1}^{\frac{m-1}2}\abs{\det g_2}^{\frac{m-3}2}\dots\abs{\det g_m}^{\frac{1-m}2}.
\]

The model $\Whit^{\fpsi_N}(\sigma)$ is a particular case of more general models considered in \cite{MR913667}
(for any reductive group). Let us recall the setup.
Let $\Lieg=\Mat_{nm,nm}$ be the Lie algebra of $G$ over $F$.
For any co-character $\varphi$ of the diagonal torus $T$ let $\Lieg=\oplus_{j\in\Z}\Lieg^\varphi_j$
be the corresponding grading
\[
\Lieg^\varphi_j=\{X\in\Lieg:\Ad(\varphi(s))X=s^jX\}
\]
and let $\Lieg^\varphi_{\ge j}=\oplus_{k\ge j}\Lieg^\varphi_k$, $j\in\Z$ be the corresponding filtration.
Let $P_\varphi$ be the semistandard parabolic subgroup such that $\Lie P_\varphi=\Lieg^\varphi_{\ge0}$.
Then $P_\varphi=M_\varphi\ltimes U_\varphi$ where $M_\varphi$ is the centralizer of $\varphi$,
$\Lie M_\varphi=\Lieg^\varphi_0$ and $\Lie U_\varphi=\Lieg^\varphi_{>0}$.
Concretely, if $\varphi(s)=\diag(s^{\lambda_1^\varphi},\dots,s^{\lambda_{mn}^\varphi})$
where $(\lambda_1^\varphi,\dots,\lambda_{mn}^\varphi)\in\Z^{mn}$ then
\begin{align*}
P_\varphi&=\{g\in G:g_{i,j}=0\text{ if }\lambda_i<\lambda_j\},\\
M_\varphi&=\{g\in G:g_{i,j}=0\text{ if }\lambda_i\ne\lambda_j\},\\
U_\varphi&=\{g\in G:g_{i,j}=\delta_{i,j}\text{ if }\lambda_i\le\lambda_j\}.
\end{align*}
Consider the nilpotent $nm\times nm$ matrix $J_{m,n}$ consisting of $m$ lower triangular Jordan blocks of size $n\times n$ each.
We say that $\varphi$ is of type $(m,n)$ if $\Ad(\varphi(s))J_{m,n}=s^{-1}J_{m,n}$, or equivalently,
if $\lambda_i^\varphi-\lambda_{i+1}^\varphi=1$ for all $i$ not dividing $n$.
If $\varphi$ is of type $(m,n)$ then $\fpsi_{U_\varphi}$ is a character of $U_\varphi$ and
by \cite{MR913667} (in particular, \S II.1) $\Hom_{U_\varphi}(\sigma,\fpsi_{U_\varphi})$ is
one-dimensional.\footnote{This is the only place in the paper where we use that $F$ is of characteristic $0$. This
assumption can probably be lifted.}
In fact, in the notation of \cite{MR913667} this is the model pertaining to the pair $(\varphi^2,J_{m,n})$.
(The more general context of \cite{MR913667} applies to cocharacters which are not necessarily even.
However, we will not discuss them here.)
We denote by $\Whit^{\fpsi_{U_\varphi}}(\sigma)$ the image of $\sigma$ in $\Ind_{U_\varphi}^G\fpsi_{U_\varphi}$.

Clearly, any $\varphi$ of type $(m,n)$ is determined by the $m$-tuple $(\lambda_n^\varphi,\lambda_{2n}^\varphi,\dots,\lambda_{mn}^\varphi)$.
We consider $(m-1)n+1$ co-characters $\varphi_0,\dots,\varphi_{(m-1)n}$ of $T$ of type $(m,n)$ such that
$\lambda_{nk}^{\varphi_i}-\lambda_{n(k+1)}^{\varphi_i}=\max(0,nk-i)$, $k=1,\dots,m-1$.
(Up to a cocharacter of the center of $G$, the cocharacter $\varphi_{(m-1)n}^2$ corresponds to the
$\operatorname{SL}_2$-triple pertaining to $J_{m,n}$.)
For simplicity we write $P_i=P_{\varphi_i}$, $M_i=M_{\varphi_i}$, $U_i=U_{\varphi_i}$.
If $i=nd+r$ where $d=\lfloor\frac in\rfloor$ and $0\le r<n$ then $U_i$ consists of the matrices whose $n\times n$ blocks $A_{j,k}$ satisfy
\begin{enumerate}
\item $A_{j,j}$ is upper unitriangular for all $j=1,\dots,m$.
\item $A_{j,k}$ is strictly upper triangular if $j\ne k$ and $j,k\le d+1$.
\item For any $k<d+2$, $(A_{d+2,k})_{a,b}=0$ if $b-a\le n-r$ and $(A_{k,d+2})_{a,b}=0$ if $a-b\ge n-r$.
\item $A_{j,k}=0$ if $j>k$ and $j>d+2$.
\end{enumerate}
(There is no constraint on $A_{j,k}$ if $j<k$ and $d+2<k$.)

In particular, $U_0=N$ while $U_{(m-1)n}$ consists of the matrices
whose difference from the identity matrix is strictly upper triangular in each $n\times n$ block.
Also, $U_{i+1}\cap N\subset U_i\cap N$ and $U_i\cap\bar N\subset U_{i+1}\cap\bar N$ for all $i$.

For brevity we write $P'=P_{(m-1)n}$, $M'=M_{(m-1)n}\simeq\overbrace{\GL_m\times\dots\times\GL_m}^n$, $U'=U_{(m-1)n}$.
In analogy with the case $m=2$ we will refer to $\Whit^{\fpsi_{U'}}(\sigma)$ as the \emph{Shalika model}.
(In general, $\Hom_{U'}(\tau,\fpsi_{U'})$ is infinite-dimensional for $\tau\in\Irr G$.)

Letting $G$ act on right on the vector space $F^{mn}$ of row vectors with standard basis $e_1,\dots,e_{mn}$,
$P'$ is the stabilizer of the flag
\[
(\spn\{e_{nj-k}:j=1,\dots,m,\ k=0,\dots,i-1\})_{i=0,\dots,n}.
\]

We denote by $\mcp:\overbrace{\GL_m\times\dots\times\GL_m}^n\rightarrow M'$ the isomorphism such that
the $i$-th copy of $\GL_m$ acts on $\spn\{e_{nj+i}:j=0,\dots,m-1\}$.

If $X$ is a matrix over $F$ then we write $\norm{X}$ for the maximum of the absolute value of its entries.

\begin{lemma} \label{lem: suppWsh}
Suppose that $\typSh\in\Whit^{\fpsi_{U'}}(\sigma)$. Then there exists $C>0$ with the following property.
Suppose that $g\in G$ with $\typSh(g)\ne0$.
Write $g=u'l'k$ where $u'\in U'$, $l'=\mcp(g_1,\dots,g_n)\in M'$ and $k\in G(\OOO)$.
Then $\norm{g_{i+1}^{-1}g_i}\le C$ for all $i<n$.
\end{lemma}

\begin{proof}
It is enough to consider the case $g\in M'$. Assume that $\typSh(g)\ne0$. Fix $1\le i<n$. For any $X\in\Mat_{m,m}(F)$ let $Y\in U'$
be the matrix such that $Y_{nj+i,nk+i+1}=X_{j+1,k+1}$ for all $0\le j,k<n$ and all other non-diagonal
entries of $Y$ are zero. Then $\typSh(gY)=\psi(\tr g_iXg_{i+1}^{-1})\typSh(g)$.
It follows that there exists $C_1>0$ depending only on $\typSh$
such that if $\typSh(g)\not=0$ then $\psi(\tr g_iXg_{i+1}^{-1})=\psi(\tr g_{i+1}^{-1}g_iX)=1$ for all $X\in\Mat_{m,m}(F)$
with $\norm{X}\le C_1$. The lemma follows.
\end{proof}

\subsection{}
We denote by $M_i^{\fpsi}$ (resp., $P_i^{\fpsi}=M_i^{\fpsi}\ltimes U_i$) the stabilizer of $\fpsi_{U_i}$ in $M_i$
(resp., $P_i$). (Note that $M_i$ determines $i$ so this
notation is unambiguous.) We also write $M'^{\fpsi}=M_{(m-1)n}^{\fpsi}$ and $P'^{\fpsi}=P_{(m-1)n}^{\fpsi}=
M'^{\fpsi}\ltimes U'$. Note that $P'^{\fpsi}$ is unimodular.
Explicitly, $M'^\fpsi$ is the image under $\mcp$ of $\GL_m$ diagonally embedded in $\overbrace{\GL_m\times\dots\times\GL_m}^n$.
It consists of the matrices in $G$ whose $n\times n$ blocks are all scalar matrices.
Let $\iota:\GL_m\rightarrow M'^\fpsi$ be the resulting identification.

In general, write $i=nd+r$, $0\le r<n$. Then the reductive part of $M_i^{\fpsi}$ is the image under $\iota$ of the subgroup
\[
\{\diag(l,t_{d+2},\dots,t_m):l\in\GL_{d+1},t_{d+2},\dots,t_m\in F^*\}.
\]
The unipotent radical of $M_i^{\fpsi}$ consists of the matrices whose $n\times n$ blocks $A_{j,k}$ satisfy
\begin{enumerate}
\item $A_{j,j}=I_n$ for all $j$.
\item If $j\ne k$ then $A_{j,k}=0$ unless $k=d+2$ and $j<k$ in which case $(A_{j,k})_{a,b}=0$ unless $a-b=n-r$.
Moreover, all the entries of $A_{j,k}$ on the diagonal $a-b=n-r$ coincide.
\end{enumerate}
It is trivial if $i$ is divisible by $n$ and it is of dimension $d+1$ otherwise.

\begin{lemma} \label{lem: ds}
Let $0\le i<(m-1)n$. Then
\begin{enumerate}
\item The commutator $[U_i,U_{i+1}]$ is contained in $U_i\cap U_{i+1}$.
Thus, $U_i\cdot U_{i+1}$ is a subgroup of $G$ which contains $U_i$ and $U_{i+1}$ as normal subgroups
and the quotients $U_iU_{i+1}/U_i\simeq U_{i+1}/U_i\cap U_{i+1}$ and $U_iU_{i+1}/U_{i+1}\simeq U_i/U_i\cap U_{i+1}$ are abelian.
Moreover,
\begin{equation} \label{eq: Uii}
U_{i+1}=(M_i\cap U_{i+1})\ltimes (U_i\cap U_{i+1})\text{ and }U_i=(M_{i+1}\cap U_i)\ltimes (U_i\cap U_{i+1}).
\end{equation}
\item We have a short exact sequence
\[
\begin{CD}
0 @>>> M_{i+1}^\fpsi\cap U_i @>>> U_i/U_i\cap U_{i+1} @>\cmap_i>> \PD(U_{i+1}/U_i\cap U_{i+1}) @>>> 0 \\
@. @. @| @| \\
@. @. M_{i+1}\cap U_i @>\cmap_i>> \PD(M_i\cap U_{i+1})
\end{CD}
\]
where $\cmap_i$ denotes the map $u\mapsto\fpsi([\cdot,u])$ and $\PD$ denotes the Pontryagin dual. Dually,
\[
\begin{CD}
0 @>>> U_{i+1}/U_i\cap U_{i+1} @>\cmap'_i>> \PD(U_i/U_i\cap U_{i+1}) @>>> \PD(M_{i+1}^\fpsi\cap U_i) @>>> 0 \\
@. @| @| \\
@. M_i\cap U_{i+1} @>\cmap'_i>> \PD(M_{i+1}\cap U_i)
\end{CD}
\]
where $\cmap_i'$ is defined by the same formula as $\cmap_i$.
\end{enumerate}
\end{lemma}

\begin{proof}
For any $j,k$ we have $\lambda_j^{\varphi_i}-\lambda_k^{\varphi_i}-(\lambda_j^{\varphi_{i+1}}-\lambda_k^{\varphi_{i+1}})\in\{-1,0,1\}$.
It follows that
\begin{equation} \label{eq: asf}
\Lieg^{\varphi_{i+1}}_{\ge j+1}\subset\Lieg^{\varphi_i}_{\ge j}\subset\Lieg^{\varphi_{i+1}}_{\ge j-1}\text{ for all }j.
\end{equation}
Therefore, $U_i\subset P_{i+1}$ and $U_{i+1}\subset P_i$.
Hence, $U_i$ and $U_{i+1}$ normalize each other, so that $U_i\cdot U_{i+1}$ is a subgroup of $G$ that contains $U_i$ and $U_{i+1}$
as normal subgroups. The equalities \eqref{eq: Uii} are now clear since $T\subset M_i,M_{i+1}$.
%Let $U_i'=M_{i+1}\cap U_i$, then $U_iU_{i+1}/U_{i+1}$ is isomorphic to $U_i'$.
By \eqref{eq: asf} we have $\Lie M_{i+1}\cap U_i=\Lieg^{\varphi_i}_{>0}\cap\Lieg^{\varphi_{i+1}}_0\subset\Lieg^{\varphi_i}_1$.
It follows that $M_{i+1}\cap U_i$ is abelian since $[\Lieg^{\varphi_i}_1,\Lieg^{\varphi_i}_1]\subset\Lieg^{\varphi_i}_2$.
Similarly $M_i\cap U_{i+1}$ is abelian.
The rest of the lemma follows easily from the fact that $U_{i+1}\cap M_i^{\fpsi}=1$.
\end{proof}

%Indeed, upon writing $i=nd+r$, $0\le r<n$,
%$M_i\cap U_{i+1}$ consists of the lower unitriangular matrices whose $n\times n$ blocks $A_{j,k}$ satisfy
%\begin{enumerate}
%\item $A_{j,j}=I_n$ for all $j$.
%\item If $j>k$ then $A_{j,k}=0$ unless $j=d+2$ in which case $(A_{j,k})_{a,b}=0$ unless
%$b-a=n-r$.
%\end{enumerate}
%We have $\dim M_i\cap U_{i+1}=(d+1)r$.
%Similarly, $U_i\cap M_{i+1}$ (resp., $U_i\cap M_{i+1}\cap\mir$, $U_i\cap M_{i+1}^{\fpsi}$)
%consists of the upper unitriangular matrices whose $n\times n$ blocks $A_{j,k}$ satisfy

\begin{remark}
If $i=nd+r$, $0\le r<n$ then $U_i\cap M_{i+1}^{\fpsi}$
consists of the upper unitriangular matrices whose $n\times n$ blocks $A_{j,k}$ satisfy
\begin{enumerate}
\item $A_{j,j}=I_n$ for all $j$.
\item If $j<k$ then $A_{j,k}=0$ unless $k=d+2$
in which case $(A_{j,k})_{a,b}=0$ unless $a-b=n-r-1$ and all
entries of $A_{j,k}$ along the diagonal $a-b=n-r-1$ are identical.
\end{enumerate}
%We have $\dim M_{i+1}\cap U_i=(d+1)(r+1)$.
%and $\dim M_{i+1}\cap U_i\cap\dir=(\lfloor\frac in\rfloor+1)(i-n\lfloor\frac in\rfloor)$.
It is of dimension $d+1$. (It coincides with the unipotent radical of $M_{i+1}^{\fpsi}$ unless $i+1$ is divisible by $n$.)
\end{remark}

\label{sec: Haar}
In the rest of the section we endow various unipotent subgroups of $G$ with Haar measures.
Thanks to the choice of basis $e_1,\dots,e_{mn}$, the Lie algebra of any of these unipotent groups
has a natural basis as a vector space over $F$. Our convention will be to take the measure corresponding
to the product measure where the Haar measure on $F$ is the one which is self-dual with respect to $\psi$.

The following is a special case of \cite{1610.00284} (see also \cite{MR3705224}). For future reference we provide the (elementary) proof.

\begin{proposition} \label{prop: trans2}
For any $i=0,\dots,(m-1)n-1$ the map
\begin{multline}\label{eq: transT}
W_i\mapsto\int_{U_i\cap U_{i+1}\bs U_{i+1}}W_i(u'\cdot)\fpsi_{U_{i+1}}(u')^{-1}\ du'=
\int_{U_i\cap U'\bs U_{i+1}\cap U'}W_i(u'\cdot)\fpsi_{U_{i+1}}(u')^{-1}\ du'
\\=\int_{U_i\cap\bar N\bs U_{i+1}\cap\bar N}W_i(u'\cdot)\ du'
\end{multline}
defines an isomorphism $\trns_i=\trns_i^\psi:\Whit^{\fpsi_{U_i}}(\sigma)\rightarrow\Whit^{\fpsi_{U_{i+1}}}(\sigma)$.
Its inverse is given by
\begin{equation} \label{eq: invi}
W_{i+1}\mapsto\int_{U_i\cap P_{i+1}^{\fpsi}\bs U_i}W_{i+1}(u\cdot)\fpsi_{U_i}(u)^{-1}\ du.
\end{equation}
In both cases the integrands are compactly supported.
\end{proposition}

\begin{proof}
For any $W_i\in\Ind_{U_i}^G\fpsi_{U_i}$, $u\in U_i$ and $u'\in U_{i+1}$ we have $W_i(u'u)=\cmap_i(u)(u')\fpsi_{U_i}(u)W_i(u')$.
It follows from Lemma \ref{lem: ds} and the smoothness of $W_i$ that $W_i\rest_{U_{i+1}}$ is compactly
supported modulo $U_i\cap U_{i+1}$ and that
for any $u\in U_i$,
\begin{equation} \label{eq: FT}
\text{$\fpsi_{U_i}(u)^{-1}\trns_iW_i(u)$ is the Fourier transform of
the function $W_i\fpsi_{U_{i+1}}^{-1}\rest_{U_{i+1}/U_i\cap U_{i+1}}$ at $\cmap_i(u)$.}
\end{equation}

Note that by multiplicity one, any $W_{i+1}\in\Whit^{\fpsi_{U_{i+1}}}(\sigma)$ is left-invariant under any unipotent
subgroup of $M_{i+1}^{\fpsi}$. Also, $U_i\cap P_{i+1}^{\fpsi}=(U_i\cap M_{i+1}^{\fpsi})\ltimes(U_i\cap U_{i+1})$.
By a similar reasoning as before, $W_{i+1}\rest_{U_i}$ is compactly supported modulo $U_i\cap P_{i+1}^{\fpsi}$.
%for any $W_{i+1}\in\Ind_{U_{i+1}}^G\fpsi_{U_{i+1}}$.
By Lemma \ref{lem: ds} and Fourier inversion, the map \eqref{eq: invi} defines a $G$-equivariant left inverse to $\trns_i$.
Since the spaces are irreducible, it is also a right inverse.
\end{proof}

\begin{remark} \label{rem: unrmtrns}
Suppose that $\sigma$ is unramified, $\psi$ has conductor $\OOO$ and $W_i\in\Whit^{\fpsi_{U_i}}(\sigma)$
is the unramified vector such that $W_i(e)=1$. Then $\trns_iW_i(e)=1$.
This follows immediately from the proof of Proposition \ref{prop: trans2}.
\end{remark}

%\begin{remark}
%The injectivity of $\trns_i$ is a special case of the results of \cite{MR3705224}.
%\end{remark}

%\Erez{Compare with \cite{MR3705224}}

%\Erez{Mention that we can prove directly that the map is an isomorphism. What does it entail?}

We write
\[
\trns=\trns^\psi=\trns_{(m-1)n-1}\circ\dots\circ\trns_0:\Whit^{\fpsi_N}(\sigma)\rightarrow\Whit^{\fpsi_{U'}}(\sigma).
\]
This operator was also considered in \cite{1802.02637}.

\subsection{}
We now introduce a subgroup of $G$ which would play an important role in what follows.
Let $\mir=\mir_{m,n}$ be the joint stabilizer of the vectors $e_{jn}$, $j=1,\dots,m$ in $G$
and let $N_\mir=N\cap\mir\supset N_M$. Note that $U'\subset\mir$.
(In the case $m=1$, $\mir$ is the standard mirabolic subgroup.)

The following is straightforward.

\begin{lemma} \label{lem: splitD}
We have $M_{i+1}\cap U_i=(M_{i+1}\cap U_i\cap\mir)\times(U_i\cap M_{i+1}^{\fpsi})$.
Hence, the restriction of $\cmap_i$ to $\mir\cap U_i/\mir\cap U_i\cap U_{i+1}\simeq\mir\cap M_{i+1}\cap U_i$ is an isomorphism.
Dually, $\cmap'_i$ defines an isomorphism between $U_{i+1}/U_i\cap U_{i+1}\simeq M_i\cap U_{i+1}$
and the Pontryagin dual of $\mir\cap U_i/\mir\cap U_i\cap U_{i+1}\simeq\mir\cap M_{i+1}\cap U_i\simeq
U_i\cap M_{i+1}/U_i\cap M_{i+1}^{\fpsi}$.
\end{lemma}

Hence, we can rewrite \eqref{eq: invi} as
\[
W_{i+1}\mapsto
\int_{\mir\cap U_i\cap U_{i+1}\bs\mir\cap U_i}W_{i+1}(u\cdot)\fpsi_{U_i}(u)^{-1}\ du
=\int_{N_\mir\cap U_{i+1}\bs N_\mir\cap U_i}W_{i+1}(u\cdot)\fpsi_{U_i}(u)^{-1}\ du.
\]

\begin{lemma} \label{L: support}
Any $\typZe\in \Ind_N^G \fpsi_N$ is compactly supported on $\mir\cap \bar N$. Hence,
\[
\trns \typZe=\int_{U'\cap N\bs U'}\typZe(u'\cdot)\fpsi_{U'}(u')^{-1}\ du'=\int_{U'\cap\bar N}\typZe(u'\cdot)\ du'
=\int_{U'\cap\bar U}\typZe(u'\cdot)\ du'
\]
where the integrand is compactly supported.
\end{lemma}

\begin{proof}
Let $g=ank\in G$ with $a=\diag(a_1,\ldots, a_{nm})$, $n\in N$ and $k\in G(\OOO)$.
It is well-known and easy to prove that if $g\in\bar N$ then
$\norm{g}\le\max_{j=1,\dots,mn}\abs{\prod_{j=i}^{nm}a_j}$.
On the other hand, if is also easy to see that if $g\in\mir$ then $\abs{a_{jn}}\leq 1$ for $j=1,\ldots,m$.
Thus, if $g\in\mir$ and $\typZe(g)\ne0$ then by the support condition for Whittaker functions
we get $\abs{a_i}\le C_1$ for all $i$ where $C_1$ depends only on $\typZe$.
By the above, if moreover $g\in\bar N$ then $\norm{g}$ is bounded in terms of $\typZe$ as required.
\end{proof}

By uniqueness $W_i\in\Whit^{\fpsi_{U_i}}(\sigma)$ is $M_i^\fpsi$-equivariant with respect to some character $\chi_i$ of $M_i^{\fpsi}$.
We explicate this character.

\begin{lemma} (cf.~\cite{1802.02637}) \label{lem: equicchar}
For any $i=nd+r$, $0\le r<n$, $l\in\GL_{d+1}$ and $t_{d+2},\dots,t_m\in F^*$ we have
\begin{multline*}
\chi_i(\iota(\diag(l,t_{d+2},\dots,t_m)))\\=\omega_\pi(t_{d+2}\dots t_m\det l)
\abs{\det l}^{-{r\choose2}+{n\choose2}(m-d-1)}\abs{t_{d+2}}^{{r\choose2}(d+1)}\prod_{j=d+2}^m\abs{t_j}^{n(n-1)(\frac{m+1}2-j)}
\end{multline*}
where $\omega_\pi$ is the central character of $\pi$.
In particular, for any $\typSh\in\Whit^{\fpsi_{U'}}(\sigma)$
\[
\typSh(\iota(l)g)=\omega_\pi(\det l)\typSh(g)\ \ \forall l\in\GL_m,\ g\in G.
\]
\end{lemma}

\begin{proof}
It is enough to evaluate $\chi_i$ on an element $\iota(t)$ where $t=\diag(t_1,\dots,t_m)$ is in the diagonal torus of $\GL_m$.
Note that $\iota(t)$ lies in the center $Z_M$ of $M$.
Writing $W_i=\int_{\bar N\cap U_i}\typZe(u\cdot)\ du$ with $\typZe\in \Whit^{\fpsi_{N}}(\sigma)$
(the integrand is compactly supported by Lemma~\ref{L: support}), the required relation follows from the equality
\begin{multline*}
\delta_P^{\frac12}\delta'^{-1}(\iota(t))=\prod_{j=1}^m\abs{t_j}^{n(n-1)(\frac{m+1}2-j)}\\=
\big(\prod_{j=1}^{d+1}\abs{t_j}^{-{r\choose2}+{n\choose2}(m-d-1)}\big)
\big(\prod_{j=d+2}^m\abs{t_j}^{n(n-1)(\frac{m+1}2-j)}\big)\abs{t_{d+2}}^{{r\choose2}(d+1)}
\delta_{Z_M\ltimes(U_i\cap\bar N)}(\iota(t))^{-1}.
\end{multline*}
\end{proof}

\begin{remark}\label{rem: w0}
Let $w_m=\left(\begin{smallmatrix}&&&1\\&&1&\\&\iddots&&\\1&&&\end{smallmatrix}\right)\in\GL_m$
and $\ww=\iota(w_m)$. By Lemma~\ref{lem: equicchar} we have
\begin{equation}\label{eq: w0}
\typSh(\ww g)= \omega_\pi(-1)^{m\choose2}\typSh(g)
\end{equation}
for any $\typSh\in \Whit^{\fpsi_{U'}}(\sigma)$.
\end{remark}

\begin{lemma} \label{L: support2}
The inverse of $\trns$ is given by
\begin{equation}\label{eq: inverse}
\typSh\mapsto\int_{N\cap U'\bs N_\mir}\typSh(u\cdot)\fpsi_N(u)^{-1}\ du
\end{equation}
where the integrand is compactly supported.
\end{lemma}

\begin{proof}
From Proposition~\ref{prop: trans2} we only need to check that the integrand is compactly supported.
Assume that $\typSh=\trns\typZe$. By \eqref{eq: w0} the integral equals
\begin{multline*}
\omega_\pi(-1)^{m\choose2} \int_{N\cap U'\bs N_\mir}\typSh(\ww u\cdot)\fpsi_N(u)^{-1}\ du
\\=\omega_\pi(-1)^{m\choose2} \int_{U\cap U'\bs U_\mir}(\int_{U'\cap \bar U}
\typZe(v \ww u\cdot)\fpsi_{U'}(v)^{-1}\fpsi_N(u)^{-1}\ dv)\ du.
\end{multline*}
where $U_\mir=U\cap\mir$.
The latter double integral is
\[
\omega_\pi(-1)^{m\choose2} \int_{\bar U\cap U'\bs \bar U_\mir}(\int_{U'\cap \bar U}
\typZe(v \bar u\ww \cdot)\fpsi_{U'}(v)^{-1}\fpsi_N(\ww\bar u\ww)^{-1}\ dv)\ d\bar u.
\]
By Lemma~\ref{L: support} the integrand is compactly supported in $v$, $\bar u$.
Thus, the integrand on the right-hand side of \eqref{eq: inverse} is compactly supported.
\end{proof}

\subsection{}
The following is an analogue of \cite[Prospotion 2]{MR0404534}
\begin{lemma} \label{lem: compind}
Any non-zero $\mir$-invariant subspace of $\Ind_{U'}^{\mir}\fpsi_{U'}$ contains $\ind_{U'}^{\mir}\fpsi_{U'}$.
In particular, $\ind_{U'}^{\mir}\fpsi_{U'}$ is irreducible.
\end{lemma}
The proof of \cite{MR0404534} (for the case $m=1$) applies word by word. One only needs to observe
that the unipotent radical $V$ of $\mir$ is abelian, the stabilizer of the character
$\fpsi_V$ under the action of $\mir$ modulo $V$ is isomorphic to $\mir_{m,n-1}$ and
the map $p\mapsto\fpsi_{V}(p^{-1}\cdot p)$ defines an open map $\mir\rightarrow\widehat{V}$

Let $Q$ be the stabilizer of $\spn\{e_{ni}:i=1,\dots,m\}$ in $G$ -- a maximal parabolic subgroup of $G$ of type $((n-1)m,m)$.
Thus, $Q=\mir\rtimes M'^{\fpsi}$ and $\delta_Q\rest_{M'^{\fpsi}}\equiv1$.

\begin{corollary} \label{cor: compind}
For any \hmgns\ $\sigma\in\Irr G$ the image of the restriction map
\begin{equation} \label{eq: restomir}
\typSh\mapsto\typSh\rest_{\mir},\ \ \Whit^{\fpsi_{U'}}(\sigma)\rightarrow\Ind_{U'}^{\mir}\fpsi_{U'}
\end{equation}
contains $\ind_{U'}^{\mir}\fpsi_{U'}$.
Equivalently, if $\sigma=\Sp(\pi,m)$, the image $\Kir^\psi(\sigma)$ of the restriction map
\[
\typSh\mapsto\typSh\rest_Q,\ \ \Whit^{\fpsi_{U'}}(\sigma)\rightarrow\Ind_{P'^{\fpsi}}^Q\omega_\pi^{\fpsi}
\]
contains $\ind_{P'^{\fpsi}}^Q\omega_\pi^\fpsi$ where $\omega_\pi^{\fpsi}$ is the character of $P'^{\fpsi}$
such that $\omega_\pi^{\fpsi}\rest_{U'}=\fpsi_{U'}$ and $\omega_\pi^{\fpsi}\circ\iota=\omega_\pi\circ\det$.
\end{corollary}

We will call $\Kir^\psi(\sigma)$ the Kirillov--Shalika model of $\sigma$.

\begin{lemma} \label{lem: compindi}
For any $i=0,\dots,(m-1)n-1$, the map
\[
\tilde\trns_i:W_i\mapsto\int_{U_i\cap U'\bs U_{i+1}\cap U'}W_i(u'\cdot)\fpsi_{U_{i+1}}(u')^{-1}\ du'
=\int_{U_i\cap\bar N\bs U_{i+1}\cap\bar N}W_i(u'\cdot)\ du'
\]
is an isomorphism between $\Ind_{\mir\cap U_i}^\mir\fpsi_{U_i}$ and
$\Ind_{\mir\cap U_{i+1}}^\mir\fpsi_{U_{i+1}}$, whose inverse is given by
\[
W_{i+1}\mapsto
\int_{\mir\cap U_i\cap U_{i+1}\bs\mir\cap U_i}W_{i+1}(u\cdot)\fpsi_{U_i}(u)^{-1}\ du
=\int_{N_\mir\cap U_{i+1}\bs N_\mir\cap U_i}W_{i+1}(u\cdot)\fpsi_{U_i}(u)^{-1}\ du.
\]
Moreover,
\[
\tilde\trns_i(\ind_{\mir\cap U_i}^\mir\fpsi_{U_i})=\ind_{\mir\cap U_{i+1}}^\mir\fpsi_{U_{i+1}}.
\]
Finally, the $\mir$-module $\ind_{\mir\cap U_i}^\mir\fpsi_{U_i}$ is irreducible.
\end{lemma}

\begin{proof}
Let $W_i\in\Ind_{\mir\cap U_i}^\mir\fpsi_{U_i}$.
As in the proof of Proposition \ref{prop: trans2}, by Lemma \ref{lem: splitD}
the function
\[
u\in\mir\cap U_i\cap U_{i+1}\bs\mir\cap U_i\mapsto\fpsi_{U_i}(u)^{-1}\trns_iW_i(u)
\]
is the Fourier transform of the function $W_i\fpsi_{U_{i+1}}^{-1}\rest_{U_i\cap U'\bs U_{i+1}\cap U'}$ at $\cmap_i(u)$.
The first claim follows by Fourier inversion.

Suppose that $W_i\in \ind_{\mir\cap U_i}^\mir\fpsi_{U_i}$.
From the definition (and since $U'\subset\mir$) $\tilde\trns_iW_i$ is supported on $(\mir\cap U_i\cdot U_{i+1})\Omega$
where $\Omega$ is a compact subset of $\mir$.
Fix $g\in \Omega$. It follows from the above that the function $\tilde\trns_iW_i(g)$ is compactly supported modulo $\mir\cap U_i\cap U_{i+1}$.
Hence, the $\tilde\trns_iW_i$ is compactly supported modulo $\mir\cap U_{i+1}$.

The last part now follows from the fact that $\ind_{U'}^\mir\fpsi_{U'}$ is irreducible.
\end{proof}

From Lemma \ref{lem: compindi}, Proposition  \ref{prop: trans2} and Corollary \ref{cor: compind} we obtain
\begin{corollary}\label{C: restomir2}
For any \hmgns\ $\sigma\in\Irr G$ and for any $i=0,\ldots,(m-1)n$, the image of the restriction map
\begin{equation} \label{eq: restgen}
W_i\mapsto W_i\rest_{\mir},\ \ \Whit^{\fpsi_{U_i}}(\sigma)\rightarrow\Ind_{U_i\cap\mir}^{\mir}\fpsi_{U_i}
\end{equation}
contains $\ind_{U_i\cap\mir}^{\mir}\fpsi_{U_i}$.
\end{corollary}

Once again, in analogy with the case $m=1$ (conjectured in \cite{MR0404534}, proved in \cite{MR0579172})
it is natural to make the following
\begin{conjecture} \label{conj: injvtve}
For any \hmgns\ $\sigma\in\Irr G$ the restriction map \eqref{eq: restomir} (or equivalently, \eqref{eq: restgen}) is injective.
\end{conjecture}
We will prove a special case in Corollary \ref{cor: inj} below.

We do not know whether in general, the restriction of $\sigma$ to $Q$ is of finite length.
(See Proposition \ref{prop: structrest} for a very special case.)
Recall that in the case $m=1$ this is known (for any $\pi\in\Irr$, not necessarily generic)
using the theory of derivatives of Bernstein--Zelevinsky \cite{MR0579172}.
It would be very interesting to have an analogous theory for $m>1$.

\section{Unitary structure} \label{sec: unit}

We take the Haar measure on $\GL_r$ to be the normalized Tamagawa measure with respect to $\psi$, i.e.
$\prod_{i=1}^r(1-q^{-i})^{-1}$ times the Haar measure associated to the standard gauge form
$\frac{\bigwedge_{i,j=1,\dots,r}dg_{i,j}}{\det g^r}$ on $\GL_r$ and the self-dual Haar measure on $F$ with respect to $\psi$.
Following our convention on Haar measures for unipotent groups (see \S\ref{sec: Haar}) we obtain a (right) Haar measure
on the $F$-points of any algebraic group whose reductive part is a product of $\GL_r$'s.
This will cover all algebraic groups considered here.

Let $\pi'\in\Irr_{\gen}\GL_n$ be another irreducible generic representation of $\GL_n$ and let $\sigma'=\Sp(\pi',m)$.

For any $0\le i\le (m-1)n$ and $s\in\C$ we define a bilinear form on
$\Whit^{\fpsi_{U_i}}(\sigma)\times\Whit^{\fpsi_{U_i}^{-1}}(\sigma')$ by
\[
\Bil_i(W_i,W_i',s)=\int_{\mir\cap U_i\bs\mir}W_i(g)W_i'(g)\abs{\det g}^s\ dg
\]
(assuming convergent).
In particular, for $\typZe\in\Whit^{\fpsi_N}(\sigma)$, $\typZe'\in\Whit^{\fpsi_N^{-1}}(\sigma')$
\[
\Bil_0(\typZe,\typZe',s)=\int_{N_\mir\bs\mir}\typZe(g)\typZe'(g)\abs{\det g}^s\ dg,
\]
and for $\typSh\in\Whit^{\fpsi_{U'}}(\sigma)$, $\typSh'\in\Whit^{\fpsi_{U'}^{-1}}(\sigma')$
\[
\Bil_{\Sh}(\typSh,\typSh',s):=\Bil_{(m-1)n}(\typSh,\typSh',s)=\int_{U'\bs\mir}\typSh(g)\typSh'(g)\abs{\det g}^s\ dg.
\]

It follows from Lemma \ref{lem: suppWsh} that $\abs{\det}$ is bounded above on the support of $\typSh\rest_{\mir}$.
Hence, if $\Bil_{\Sh}(\typSh,\typSh',s)$ converges absolutely at $s_0\in\R$
then it converges absolutely for any $s$ with $\Re s\ge s_0$.
A similar statement holds for any $\Bil_i$ although we will not use it.

We also write $\Bil_i(W_i,W_i')=\Bil(W_i,W_i',0)$ assuming the latter is well-defined
(either as a convergent integral, or by analytic continuation), in which case it is $\mir$-invariant.

In general, we do not know whether $\Bil_i(\cdot,\cdot)$ is always defined.

\begin{proposition} \label{prop: unit}
The integral defining $\Bil_i(W_i,W_i',s)$ converges for $\Re s+\expo(\pi)+\expo(\pi')+1>0$.
Moreover, for all $0\le i<(m-1)n$, $W_i\in\Whit^{\fpsi_{U_i}}(\sigma)$,
$W_i'\in\Whit^{\fpsi_{U_i}^{-1}}(\sigma')$ we have
\begin{equation} \label{eq: Bilii+1}
\Bil_{i+1}(\trns_i^\psi W_i,\trns_i^{\psi^{-1}}W_i',s)=\Bil_i(W_i,W_i',s).
\end{equation}
Finally, there exist $W_i\in\Whit^{\fpsi_{U_i}}(\sigma)$ and
$W_i'\in\Whit^{\fpsi_{U_i}^{-1}}(\sigma')$ such that $\Bil_i(W_i,W_i',s)\equiv1$ for all $s\in\C$.
\end{proposition}

\begin{remark}
In the next section we prove that $\Bil_i(W_i,W_i',s)$ admits meromorphic continuation in $s$ to a rational function in $q^s$.
\end{remark}

\begin{proof}
First note that the last statement follows from Corollary \ref{C: restomir2}.

Next, we show the convergence of the integral defining $\Bil_0$.
Upon twisting $\pi$ and $\pi'$ by $\abs{\cdot}^{(s+\expo(\pi')-\expo(\pi))/2}$
and $\abs{\cdot}^{(s+\expo(\pi)-\expo(\pi'))/2}$ respectively and using the inequality
$\abs{xy}\le(\abs{x}^2+\abs{y^2})/2$
we may assume without loss of generality that $\pi'=\overline{\pi}$, $\typZe'=\overline{\typZe}$ and $s=0$.
Thus, we need to show the convergence of
\[
\int_{N_\mir\bs\mir}\abs{\typZe(g)}^2\ dg
\]
provided that $\expo(\pi)>-\frac12$.
In fact, we show a slightly stronger assertion, namely the convergence of
\begin{equation} \label{eq: strconv}
\int_{\mir\bs G}\int_{N_\mir\bs\mir}\abs{\typZe(lg)}^2\ dl\ \Phi(\eta g)\abs{\det g}^m\ dg
\end{equation}
for any $0\le\Phi\in\swrz(\Mat_{m,nm}(F))$
where $\eta\in\Mat_{m,nm}(F)$ is the matrix whose $i$-th row is $e_{ni}$, $i=1,\dots,m$.
Note that the stabilizer of $\eta$ under the right $G$-action on $\Mat_{m,nm}(F)$ is $\mir$.
Since the modulus character of $\mir$ is $\abs{\det}^m$, \eqref{eq: strconv} is formally well-defined
and can be rewritten as
\begin{multline} \label{eq: dcmp}
\int_{N_\mir\bs G}\abs{\typZe(g)}^2\Phi(\eta g)\abs{\det g}^{m}\ dg\\=
\int_{P\bs G}\int_{N\bs P}\abs{\typZe(lg)}^2\int_{N_\mir\bs N}\Phi(\eta ulg)\ du
\ \abs{\det l}^m\delta_P(l)^{-1}\ dl\ \abs{\det g}^m\ dg\\=
\int_{P\bs G}\int_{N_M\bs M}\abs{\typZe(lg)}^2\int_{U_\mir\bs U}\Phi(\eta ulg)\ du
\ \abs{\det l}^m\delta_P(l)^{-1}\ dl\ \abs{\det g}^m\ dg.
\end{multline}
We may identify the vector space $\Mat_{m,nm}(F)$ with $\Mat_{m,m}(F^n)$.
Observe that for any $l=\diag(g_1,\dots,g_m)\in M$, $g\in G$ we have
\begin{equation} \label{eq: relPhiPsi}
\abs{\det l}^{\frac{m-1}2}\int_{U_\mir\bs U}\Phi(\eta ulg)\ du=\tilde\Phi_g(e_ng_1,\dots,e_ng_m)\delta'(l)
\end{equation}
where $\tilde\Phi_g\in\swrz((F^n)^m)$ is the function
\begin{equation} \label{def: Psi}
\tilde\Phi_g(v_1,\dots,v_m)=\int\Phi(Xg)\ dX,\ \ v_1,\dots,v_m\in F^n,
\end{equation}
where the integral is taken over the $n{m\choose 2}$-dimensional affine space of upper triangular $F^n$-valued $m\times m$-matrices
whose diagonal entries are $v_1,\dots,v_m$.
Thus, \eqref{eq: dcmp} is equal to
\begin{multline*}
\int_{P\bs G}\int_{(N_n\bs\GL_n)^m}\abs{\delta_P^{-\frac12}\delta'(l)\typZe(lg)}^2
\tilde\Phi_g(e_ng_1,\dots,e_ng_m)\prod_{i=1}^m\abs{\det g_i}^i\ dg_1\dots dg_m
\\\abs{\det g}^m\ dg
\end{multline*}
where $l=\diag(g_1,\dots,g_m)\in M$.
Thus, by \eqref{eq: restW} the inner integral is
a finite linear combination of products of Rankin--Selberg integrals for $\pi\times\bar\pi$ at $i$, $i=1,\dots,m$.
The assumption that $\expo(\pi)>-\frac12$ guarantees that these Rankin--Selberg integrals converge.
Since the outer integral is a finite sum, we obtain the convergence of \eqref{eq: strconv}.

Now let $0\le i<(m-1)n$, $W_i\in\Whit^{\fpsi_{U_i}}(\sigma)$, $W_i'\in\Whit^{\fpsi_{U_i}^{-1}}(\sigma')$.
Recall that $W_i\rest_{U_{i+1}}$ (resp., $\trns_i^\psi W_i\rest_{U_i\cap\mir}$) is compactly supported modulo
$U_i\cap U_{i+1}$ (resp., $U_i\cap U_{i+1}\cap\mir$).

Moreover, by the unitarity of Fourier transform and the argument of Proposition \ref{prop: trans2}
(cf. \eqref{eq: FT}) we have
\begin{multline} \label{eq: unitstep}
\int_{\mir\cap U_{i+1}\bs\mir\cap U_iU_{i+1}}\trns_i^\psi W_i(u)\trns_i^{\psi^{-1}}W_i'(u)\ du=
\int_{N_\mir\cap U_{i+1}\bs N_\mir\cap U_i}\trns_i^\psi W_i(u)\trns_i^{\psi^{-1}}W_i'(u)\ du\\=
\int_{U_i\bs U_iU_{i+1}}W_i(u)W_i'(u)\ du=
\int_{\mir\cap U_i\bs\mir\cap U_iU_{i+1}}W_i(u)W_i'(u)\ du
\end{multline}
where the integrals are absolutely convergent.
(We can also write the integrals as
\[
\int_{U_i\cap U'\bs U_{i+1}\cap U'}W_i(u)W_i'(u)\ du=
\int_{U_i\cap\bar N\bs U_{i+1}\cap\bar N}W_i(u)W_i'(u)\ du.)
\]
It follows that if at least one of integrals
\[
\int_{\mir\cap U_i\bs\mir}\abs{W_i(g)}^2+\abs{W_i'(g)}^2\ dg
\text{ or }
\int_{\mir\cap U_{i+1}\bs\mir}\abs{\trns_iW_i(g)}^2+\abs{\trns_iW_i'(g)}^2\ dg
\]
converges then so is the other and %$\Bil_{i+1}(\trns_i^\psi W_i,\trns_i^{\psi^{-1}}W_i')=\Bil_i(W_i,W_i')$.
\begin{multline*}
\Bil_{i+1}(\trns_i^\psi W_i,\trns_i^{\psi^{-1}}W_i')=
\int_{\mir\cap U_iU_{i+1}\bs\mir}\int_{\mir\cap U_{i+1}\bs\mir\cap U_iU_{i+1}}
\trns_i^\psi W_i(ug)\trns_i^{\psi^{-1}}W_i'(ug)\ du\ dg\\=
\int_{\mir\cap U_iU_{i+1}\bs\mir}\int_{\mir\cap U_i\bs\mir\cap U_iU_{i+1}}W_i(ug)W_i'(ug)\ du\ dg=
\Bil_i(W_i,W_i').
\end{multline*}

We can now conclude the convergence for all $i$
and the identity \eqref{eq: Bilii+1} since they clearly reduce to the case $s=0$.
\end{proof}

\begin{proposition} \label{prop: inv}
Suppose that $\pi'=\pi^\vee$ (or equivalently, $\sigma'=\sigma^\vee$) and $\pi$ is \AT.
Then $\Bil_i(W_i,W_i')$ is a well-defined $G$-invariant pairing on
$\Whit^{\fpsi_{U_i}}(\sigma)\times\Whit^{\fpsi_{U_i}^{-1}}(\sigma^\vee)$.
\end{proposition}

\begin{proof}
By Proposition \ref{prop: unit} $\Bil_i(\cdot,\cdot)$ is well-defined and not identically zero.
To show invariance it suffices to consider $i=0$.
We use induction on $m$.
The case $m=1$ (in which $\mir$ is the standard mirabolic subgroup) is well known and follows from Bernstein's theorem \cite{MR748505}.
For the induction step, let $m>1$ and let $Q'$ be the subgroup of
the standard maximal parabolic subgroup of $G$ of type $((m-1)n,n)$
consisting of the matrices whose lower right $n\times n$ corner is upper unitriangular.
Write
\begin{multline*}
\Bil_0(\typZe,\typZe')=\int_{\mir\cap Q'\bs\mir}\int_{N_\mir\bs\mir\cap Q'}\typZe(qg)\typZe'(qg)\abs{\det q}^{1-n}\ dq\ dg\\=
\int_{\mir\cap Q'\bs\mir}\int_{\mir_{m-1,n}\cap N_{m-1,n}\bs\mir_{m-1,n}}\typZe(qg)\typZe'(qg)\abs{\det q}^{1-n}\ dq\ dg.
\end{multline*}
Here we consider $\GL_{(m-1)n}$ (and hence, $\mir_{m-1,n}$) as a subgroup of $G$.
(Note that $\delta_{\mir}=\abs{\det}^m$ while $\delta_{\mir\cap Q'}=\abs{\det}^{n+m-1}$.)
By \eqref{eq: maxrest} and the induction hypothesis, the inner integral is left $(Q',\abs{\det}^{n-1})$-equivariant in $g$.
Hence, we can replace the domain of outer integration by $Q'\bs\mir_{1,mn}$
where $\mir_{1,mn}$ is the standard mirabolic subgroup of $G$ (the stabilizer of $e_{mn}$).
(Note that $\delta_{\mir_{1,mn}}=\abs{\det}$ and $\delta_{Q'}=\abs{\det}^n$.)
It follows that $\Bil_0(\cdot,\cdot)$ is $\mir_{1,mn}$-invariant. By Bernstein's theorem, it is $G$-invariant as required.
\end{proof}

%Alternatively, can we use the argument of Jacquet-Shalika?

We immediately deduce a special case of Conjecture \ref{conj: injvtve}.
\begin{corollary} \label{cor: inj}
Conjecture \ref{conj: injvtve} holds for any $\pi\in\Irr_{\AT}\GL_n$.
In particular, it holds for any unitarizable $\pi\in\Irr_{\gen}\GL_n$.
\end{corollary}

\begin{remark}
By analytic continuation it is easy to prove Conjecture \ref{conj: injvtve}
for $\pi$ of the form $\pi=\tau_1\abs{\cdot}^{\lambda_1}\times\dots\times\tau_k\abs{\cdot}^{\lambda_k}$
where $\tau_i\in\Irr_{\sqr}$ are fixed and $(q^{\lambda_1},\dots,q^{\lambda_k})$ is in general position.
\end{remark}

In view of Proposition \ref{prop: inv} and Bernstein's theorem it is natural to make following related conjecture.

\begin{conjecture}
For any \hmgns\ $\sigma\in\Irr G$, every $\mir$-invariant bilinear form on $\sigma\times\sigma^\vee$ is $G$-invariant.
\end{conjecture}
Perhaps even more is true.
\begin{conjecture}
For any \hmgns\ $\sigma,\sigma'\in\Irr G$, there is a unique up to scalar $\mir$-invariant bilinear form on $\sigma\times\sigma'$.
\end{conjecture}
(We do not know whether this is known in the case $m=1$.)

\section{Local zeta integrals} \label{sec: locint}

Throughout this section let $\pi$, $\pi'\in\Irr_{\gen}\GL_n$ and $\sigma=\Sp(\pi,m),\sigma'=\Sp(\pi',m)\in\Irr G$.

\subsection{}
Recall the $\GL_n\times\GL_n$ local Rankin-Selberg integrals studied by Jacquet--Piatetski-Shapiro--Shalika \cite{MR701565}.
They are given by
\[
Z^{\GL_n}(W,W',\Phi,s)=\int_{N_n\bs\GL_n}W(g)W'(g)\Phi(e_ng)\abs{\deg g}^s\ dg
\]
where $W\in\Whit^{\fpsi_{N_n}}(\pi)$, $W'\in\Whit^{\fpsi_{N_n}^{-1}}(\pi')$, $\Phi\in\swrz(F^n)$ and $s\in\C$.
The integral converges for $\Re s+\expo(\pi)+\expo(\pi')>0$ and admits a meromorphic continuation in $s$
to a rational function in $q^s$.
The quotient
\[
\frac{Z^{\GL_n}(W,W',\Phi,s)}{L(s,\pi\times\pi')}
\]
is a Laurent polynomial in $q^{\pm s}$ which can be made non-zero at any given $s\in\C$ by an appropriate choice of
$W$, $W'$, $\Phi$. Moreover, we have a functional equation
\[
Z^{\GL_n}(\widehat{W},\widehat{W'},\hat\Phi,1-s)=\omega_{\pi'}(-1)^{n-1}\gamma(s,\pi\times\pi',\psi)Z^{\GL_n}(W,W',\Phi,s)
\]
where $\widehat{W}\in\Whit^{\fpsi_{N_n}^{-1}}(\pi^\vee)$, $\widehat{W'}\in\Whit^{\fpsi_{N_n}}(\pi'^\vee)$ are given by
\[
\widehat{W}(g)=W(w_n\,^tg^{-1}),\ \widehat{W'}(g)=W'(w_n\,^tg^{-1})
%,\ \ w_n=\left(\begin{smallmatrix}&&&1\\&&1&\\&\iddots&&\\1&&&\end{smallmatrix}\right)
\]
and $\hat\Phi$ is the Fourier transform of $\Phi$ given by
\[
\hat\Phi(y)=\int_{F^n}\Phi(x)\psi(\sprod xy)\ dx
\]
where $\sprod{(x_1,\dots,x_n)}{(y_1,\dots,y_n)}=\sum_ix_iy_i$ denotes the standard pairing on $F^n$.

Slightly more generally,
for $W\in\Whit^{\fpsi_{N_M}}(\pi^{\otimes m})$, $W'\in\Whit^{\fpsi_{N_M}^{-1}}(\pi'^{\otimes m})$,
$\tilde\Phi\in\swrz((F^n)^m)$ and $(s_1,\dots,s_m)\in\C^m$ we write
\[
Z^M(W,W',\tilde\Phi,(s_1,\dots,s_m))=\int_{N_M\bs M}W(l)W'(l)\tilde\Phi(e_ng_1,\dots,e_ng_n)\prod_{i=1}^m\abs{\det g_i}^{s_i}\ dl
\]
where $l=\diag(g_1,\dots,g_m)\in M$. This is a linear combination of products
$\prod_{i=1}^mZ^{\GL_n}(W_i,W_i',\Phi_i,s_i)$ where $W_i\in\Whit^{\fpsi_{N_n}}(\pi)$,
$W_i'\in\Whit^{\fpsi_{N_n}^{-1}}(\pi')$ and $\Phi_i\in\swrz(F^n)$.
Thus,
\begin{multline} \label{eq: convZM}
\text{the integral defining $Z^M(W_1,W_2,\Phi,(s_1,\dots,s_m))$ is absolutely convergent}
\\\text{provided that $\Re s_i+\expo(\pi)+\expo(\pi')>0$ for all $i$.}
\end{multline}
Moreover, we have a functional equation
\begin{multline} \label{eq: FEM}
Z^M(\widehat W^M,\widehat W'^M,\widehat{\tilde\Phi}^M,(1-s_1,\dots,1-s_m))=
\omega_{\pi'}(-1)^{m(n-1)}\prod_{i=1}^m\gamma(s_i,\pi\times\pi',\psi)\\
Z^M(W,W',\tilde\Phi,(s_1,\dots,s_m))
\end{multline}
where $\widehat W^M(l)=W(\diag(\overbrace{w_n,\dots,w_n}^m)\,^tl^{-1})$ and
\[
\widehat{\tilde\Phi}^M(X_1,\dots,X_m)=\int_{(F^n)^m}\Phi(Y_1,\dots,Y_m)\psi(\sum_i\sprod{X_i}{Y_i})\ dY_1\dots\ dY_m.
\]

\subsection{}
We write an analog of the Rankin-Selberg integral for $\sigma\times\sigma'$ on the Shalika model as follows.
Recall that $\eta\in\Mat_{m,nm}(F)$ is the matrix whose $i$-th row is $e_{ni}$, $i=1,\dots,m$,
so that $\mir$ is the stabilizer of $\eta$ in $G$.
For any $\typSh\in\Whit^{\fpsi_{U'}}(\sigma)$, $\typSh'\in\Whit^{\fpsi_{U'}^{-1}}(\sigma')$, $\Phi\in\swrz(\Mat_{m,nm}(F))$
consider
\[
Z(\typSh,\typSh',\Phi,s)=\int_{U'\bs G}\typSh(g)\typSh'(g)\Phi(\eta g)\abs{\det g}^s\ dg.
\]

This expression was already considered in some form in the proof of Proposition \ref{prop: unit}.

Note that in the case $n=1$ (where $U'=1$) $Z$ reduces to the generalized Tate integral for $\GL_m$
considered by Godement--Jacquet.

In general, one can relate the integrals $Z(\typSh,\typSh',\Phi,s)$ to the usual Rankin-Selberg integrals as follows.

\begin{proposition} \label{prop: intermsofZe}
For any $\typZe\in\Whit^{\fpsi_N}(\sigma)$, $\typZe'\in\Whit^{\fpsi_N^{-1}}(\sigma')$
and $\Phi\in\swrz(\Mat_{m,nm}(F))$ we have
\begin{multline} \label{eq: ZtoZM}
Z(\trns^\psi\typZe,\trns^{\psi^{-1}}\typZe',\Phi,s)=\\
\int_{P\bs G}Z^M((\typZe)_g,(\typZe')_g,\tilde\Phi_g,(s-m+1,\dots,s))\abs{\det g}^s\ dg
\end{multline}
where $(\typZe)_g=\delta_P^{-\frac12}\delta'\typZe(\cdot g)\in\Whit^{\fpsi_{N_M}}(\pi^{\otimes m})$,
$(\typZe')_g=\delta_P^{-\frac12}\delta'\typZe'(\cdot g)\in\Whit^{\fpsi_{N_M}^{-1}}(\pi'^{\otimes m})$
and $\tilde\Phi_g$ is given by \eqref{def: Psi}.
The integral on the right-hand side, as well as the integral defining $Z(W_1,W_2,\Phi,s)$ is absolutely convergent
for $\Re s+\expo(\pi)+\expo(\pi')+1>m$.
\end{proposition}

\begin{proof}
Write $Z(\trns^\psi\typZe,\trns^{\psi^{-1}}\typZe',\Phi,s)$ as
\[
\int_{\mir\bs G}\int_{U'\bs\mir}
\trns^\psi\typZe(lg)\trns^{\psi^{-1}}\typZe'(lg)\abs{\det l}^{s-m}\ dl\ \Phi(\eta g)\abs{\det g}^s\ dg.
\]
By Proposition \ref{prop: unit} %(cf. Remark \ref{rem: gens})
we get
\begin{multline*}
\int_{\mir\bs G}\int_{N_\mir\bs\mir}
\typZe(lg)\typZe'(lg)\abs{\det l}^{s-m}\ dl\ \Phi(\eta g)\abs{\det g}^s\ dg\\=
\int_{N_\mir\bs G}\typZe(g)\typZe'(g)\Phi(\eta g)\ dg\abs{\det g}^s\ dg.
\end{multline*}
We write it as
\[
\int_{P\bs G}\int_{N_M\bs M}\typZe(lg)\typZe'(lg)\int_{U_\mir\bs U}\Phi(\eta ulg)\ du
\ \abs{\det l}^s\delta_P(l)^{-1}\ dl\ \abs{\det g}^s\ dg.
\]
The required identity now follows from \eqref{eq: relPhiPsi}.
For convergence, as in the proof of Proposition \ref{prop: unit},
we may assume that $\Phi\ge0$, $s\in\R$, $\pi'=\bar\pi$ and $W_2=\overline{W_1}$,
%(using the inequality $\abs{xy}\le\frac12(\abs{x}^2+\abs{y}^2)$).
so that all the integrands considered above are non-negative. Therefore, the manipulations are justified
for $s+2\expo(\pi)+1>m$ by \eqref{eq: convZM}.
\end{proof}

\begin{corollary} \label{cor: unram}
For any $\typSh\in\Whit^{\fpsi_{U'}}(\sigma)$, $\typSh'\in\Whit^{\fpsi_{U'}^{-1}}(\sigma')$
and $\Phi\in\swrz(\Mat_{m,nm}(F))$ the function
\[
\big(\prod_{i=0}^{m-1}L(s-i,\pi\times\pi')\big)^{-1}Z(\typSh,\typSh',\Phi,s)
\]
is a Laurent polynomial in $q^{\pm s}$, hence entire.

Moreover, if $\sigma,\sigma'$ are unramified, $\typSh\in\Whit^{\fpsi_{U'}}(\sigma)$, $\typSh'\in\Whit^{\fpsi_{U'}^{-1}}(\sigma)$
are the unramified vectors such that $\typSh(e)=\typSh(e)=1$, $\Phi$ is the characteristic function of $\Mat_{m,nm}(\OOO)$ and
$\psi$ has conductor $\OOO$ then
\[
Z(\typSh,\typSh',\Phi,s)=\prod_{i=0}^{m-1}L(s-i,\pi\times\pi')=L(s-\frac{m-1}2,\pi\times\sigma')=L(s-\frac{m-1}2,\sigma\times\pi').
\]
\end{corollary}

Indeed, this follows from the analogous statements for the usual Rankin--Selberg integrals together with
Remark \ref{rem: unrmtrns}.

\begin{remark}
In general, without the assumption that $\sigma$ and $\sigma'$ are unramified, the equality
\[
L(s-\frac{m-1}2,\pi\times\sigma')=L(s-\frac{m-1}2,\sigma\times\pi')
\]
does not always hold.
\end{remark}

For any $\typSh\in\Whit^{\fpsi_{U'}}(\sigma)$ let $\widehat\typSh\in\Whit^{\fpsi_{U'}^{-1}}(\sigma^\vee)$
be given by $\widehat\typSh(g)=\typSh(w_{nm}\,^tg^{-1})$.

\begin{corollary}
For any $\pi,\pi'\in\Irr_{\gen}\GL_n$ we have a local functional equation
\begin{equation} \label{eq: FE}
Z(\widehat{\typSh},\widehat{\typSh'},\hat\Phi,m-s)=
\omega_{\pi'}(-1)^{(n-1)m}\left(\prod_{i=0}^{m-1}\gamma(s-i,\pi\times\pi',\psi)\right)Z(\typSh,\typSh',\Phi,s)
\end{equation}
where $\widehat\Phi$ is the Fourier transform
\[
\widehat\Phi(X)=\int_{\Mat_{m,nm}(F)}\Phi(Y)\psi(\tr \,^tY w_m X)\ dY.
\]
\end{corollary}

\begin{proof}
For any $\typZe\in\Whit^{\fpsi_N}(\sigma)$ define
$\widehat\typZe\in\Whit^{\fpsi_N^{-1}}(\sigma^\vee)$ by $\widehat\typZe(g)=\typZe(w_{nm}\,^t{g}^{-1})$.
Then $\widehat{\trns\typZe}=\trns(\widehat\typZe)$.
Note $w_{mn}=\diag(\overbrace{w_n,\dots,w_n}^m)\ww$ where $\ww$ is as in Remark \ref{rem: w0}; write $g'=\ww\,^tg^{-1}$, $g\in G$.
Then for any $g\in G$ we have
\[
(\widehat\typZe)_g(l)=\widehat{(\typZe)_{g'}}^M(\ww l\ww^{-1}),\ \ l\in M
\]
and by Fourier inversion
\[
\widetilde{(\widehat\Phi)}_g(v_1,\dots,v_m)=\abs{\det g}^{-m}\widehat{\tilde\Phi_{g'}}^M(v_m,\dots,v_1),\ \ v_1,\dots,v_m\in F^n.
\]
The corollary therefore follows from Proposition \ref{prop: intermsofZe} and the functional equation \eqref{eq: FEM}
using the change of variable $g\mapsto g'$ in the integral on the right-hand side of \eqref{eq: ZtoZM}.
\end{proof}

\begin{corollary} \label{cor: absconv}
Suppose that $\pi$ is \AT\ and let $\pi'=\pi^\vee$.
Then for any $\typSh\in\Whit^{\fpsi_{U'}}(\sigma)$, $\typSh'\in\Whit^{\fpsi_{U'}^{-1}}(\sigma')$ we have
\[
Z(\typSh,\typSh',\Phi,m)=c\Bil_{\Sh}(\typSh,\typSh')\hat\Phi(0)
\]
where both sides are well-defined. Here $c$ is a constant which depends only on $F$
\end{corollary}

\begin{proof}
Since the modulus function of $\mir$ is $\abs{\det}^m$, we have
\begin{multline*}
Z(\typSh,\typSh',\Phi,m)=\int_{U'\bs G}\typSh(g)\typSh'(g)\Phi(\eta g)\abs{\det g}^m\ dg=
\\\int_{\mir\bs G}\int_{U'\bs\mir}\typSh(pg)\typSh'(pg)\ dp\ \Phi(\eta g)\abs{\det g}^m\ dg.
\end{multline*}
For $\pi'=\pi^\vee$, by Proposition \ref{prop: inv} we get
\[
\Bil_{\Sh}(\typSh,\typSh')\int_{\mir\bs G}\ \Phi(\eta g)\abs{\det g}^m\ dg=c\hat\Phi(0)\Bil_{\Sh}(\typSh,\typSh')
\]
as required.
\end{proof}

\begin{corollary} \label{cor: ors=0}
Suppose that $\pi$ is \AT\ and let $\pi'=\pi^\vee$.
Then the highest possible order of pole of
$Z(\typSh,\typSh',\Phi,s)$ at $s=0$ is the order $\ord$ of the zero of the product of $\gamma$-factors
on the right-hand side of \eqref{eq: FE} at $s=0$.
%More precisely, if $\pi=\tau_1\times\dots\times\tau_k$ with
\end{corollary}

\begin{example}
If $\pi\in\Irr_{\sqr}\GL_n$ corresponds to a segment of length $k$ and $\pi'=\pi^\vee$ then $\ord=\min(m,k)$.
Indeed, In this case $\big(\prod_{j=1}^k\frac{1-q^{f(s-j)}}{1-q^{-f(s+j-1)}}\big)\gamma(s,\pi\times\pi^\vee,\psi)$
is entire for a suitable integer $f>0$ depending on $\pi$.
\end{example}

\begin{corollary} \label{cor: polem-1}
Suppose that $\omega_\pi$ is unitary and let $\pi'=\overline{\pi}$.
Then for suitable $\typSh\in\Whit^{\fpsi_{U'}}(\sigma)$, $\typSh'\in\Whit^{\fpsi_{U'}^{-1}}(\sigma^\vee)$
and $\Phi\in\swrz(\Mat_{m,nm}(F))$, $Z(\typSh,\typSh',\Phi,s)$ has at least one pole for $\Re s\ge m-1$.
\end{corollary}

Indeed, taking $\typSh'=\overline{\typSh}$ and $\Phi\ge0$, the right-hand side of \eqref{eq: ZtoZM}
is a power series in $q^{-s}$ with non-negative coefficients $a_k$ which vanish for $k\ll0$.
Assume on the contrary that $Z(\typSh,\typSh',\Phi,s)$ is holomorphic throughout $\Re s\ge m-1$.
Then the power series would converge at $s=m-1$.
However, the integral on the right-hand side of \eqref{eq: ZtoZM} diverges at $s=m-1$
since it contains $\int_{F^*}\tilde\Phi_g(\lambda e_n,e_n,\dots,e_n)\abs{\lambda}^{s-m+1}\ d\lambda$
as an inner integral.

\subsection{} \label{sec: fphi}
Recall that $Q=\mir\rtimes M'^{\fpsi}$, $\delta_Q\rest_{\mir}=\delta_{\mir}=\abs{\det}^m$ and $\delta_Q\rest_{M'^{\fpsi}}=1$.
Hence, we can write $Z(\typSh,\typSh',\Phi,s)$ as
\[
\int_{Q\bs G}\int_{U'\bs\mir}\int_{M'^{\fpsi}}\typSh(lpg)\typSh'(lpg)\Phi(\eta lg)\abs{\det l}^s\ dl\ \abs{\det p}^{s-m}\ dp\
\ \abs{\det g}^s\ dg.
\]
Using Lemma \ref{lem: equicchar} and the identification $\iota:\GL_m\rightarrow M'^\fpsi$ we get
\begin{equation} \label{eq: Zbileq}
Z(\typSh,\typSh',\Phi,s)=\int_{Q\bs G}\Bil_{\Sh}(\typSh(\cdot g),\typSh'(\cdot g),s-m)f_{\Phi,\omega_\pi\omega_{\pi'},s}(g)\ dg
\end{equation}
where for any character $\omega$ of $F^*$
\[
f_{\Phi,\omega,s}(g)=\int_{\GL_m}\Phi'_g(l)\omega(\det l)\abs{\det l}^{ns}\ dl\ \abs{\det g}^s
\]
and $\Phi'_g\in\swrz(\Mat_{m,m}(F))$ is given by $\Phi'_g(X)=\Phi(\mu(X)g)$ where $\mu(X)\in\Mat_{m\times nm}$
is the matrix whose $i$-th row is $\sum_{j=1}^mX_{i,j}e_{nj}$.
%the columns $n,2n,\dots,mn$ of $\Phi(\cdot g)$.
Note that $\Phi\mapsto f_{\Phi,\omega,s}$ is an intertwining map from $\swrz(\Mat_{m,nm}(F))\otimes\abs{\det}^s$
to $\Ind_Q^G\nu_s$ where $\nu_s$ is the character on $Q$ such that $\nu_s\rest_{\mir}=\abs{\det}^{s-m/2}$
and $\nu_s\circ\iota=\omega^{-1}\circ\det$.

\begin{corollary}
There exist $\typSh\in\Whit^{\fpsi_{U'}}(\sigma)$, $\typSh'\in\Whit^{\fpsi_{U'}^{-1}}(\sigma')$
and $\Phi\in\swrz(\Mat_{m,nm}(F))$ such that $Z(\typSh,\typSh',\Phi,s)\equiv1$.
\end{corollary}

This follows from Corollary \ref{cor: compind} and \eqref{eq: Zbileq} by taking $\typSh$ such that
$\typSh\rest_{\mir}$ is supported in $U'\Omega$ for a small neighborhood $\Omega$ of $e$
and $\Phi$ supported in a small neighborhood of $\eta$.

Let $\ord_{\Bil}(s)=\ord_{\Bil;\sigma,\sigma'}(s)\ge0$ be the maximal order of pole
of $\Bil_i(W_i,W_i',\cdot)$ at $s$ for $i=0,\dots,(m-1)n$ as we vary $W_i\in\Whit^{\fpsi_{U_i}}(\sigma)$,
$W'_i\in\Whit^{\fpsi_{U_i}^{-1}}(\sigma')$. (Recall that this does not depend on $i$ by \eqref{eq: Bilii+1}.)
Similarly, let $\ord_Z(s)=\ord_{Z;\sigma,\sigma'}(s)\ge0$ be the maximal order of pole of
$Z(\typSh,\typSh',\Phi,\cdot)$ at $s$ as we vary $W\in\Whit^{\fpsi_{U'}}(\sigma)$, $W'\in\Whit^{\fpsi_{U'}^{-1}}(\sigma')$
and $\Phi\in\swrz(\Mat_{m,nm}(F))$.

\begin{corollary} \label{cor: cmprpole}
The bilinear form $\Bil_i(\cdot,\cdot,s)$ on $\Whit^{\fpsi_{U_i}}(\sigma)\times\Whit^{\fpsi_{U_i}^{-1}}(\sigma')$
admits meromorphic continuation in $s$ to a rational function in $q^s$.
Moreover, for every $s\in\C$ we have $\ord_{\Bil}(s-m)\le\ord_Z(s)$ with an equality unless $\omega_\pi\omega_{\pi'}=\abs{\cdot}^{j-ns}$
for some $j\in\{0,\dots,m-1\}$ in which case $\ord_Z(s)\le\ord_{\Bil}(s-m)+1$.
In particular, if $\pi'=\pi^\vee$ then $\Bil_0(\cdot,\cdot)$
is defined if and only if $Z(\cdot,\cdot,\cdot,s)$ is holomorphic at $s=m$ for all data.
\end{corollary}

\begin{proof}
It is enough to prove the meromorphic continuation for $i=(m-1)n$, i.e., for $\Bil_{\Sh}$.
This case follows from the equality \eqref{eq: Zbileq}.
Indeed, $f_{\Phi,\omega,s}(g)$ is a generalized Tate integral with respect to $\GL_m$, and hence
\[
L(ns-\frac{m-1}2,\omega\circ\det{}_{\GL_m})f_{\Phi,\omega,s}=
\big(\prod_{i=0}^{m-1}L(ns-i,\omega)\big)f_{\Phi,\omega,s}
\]
is entire.
On the other hand, if $\Phi$ is the characteristic function of a small neighborhood of $\eta$ then
$f_{\Phi,\omega,s}$ is supported in $Q\Omega$ for a small neighborhood $\Omega$ of $e$ and hence
$Z(\typSh,\typSh',\Phi,s)$ is a constant multiple of $\Bil_{\Sh}(\typSh,\typSh',s-m)$.
The corollary follows.
\end{proof}

\begin{remark}
Note that if $\pi$ and $\pi'$ are tempered then $\ord_{Z}(s)=0$ unless $\Re s\in\frac12\Z$ and $\Re s\le m$.
Thus, in general, many poles of $f_{\Phi,\omega,s}$ do not contribute a pole for $Z(\cdot,\cdot,\cdot,s)$.
\end{remark}

\begin{corollary} \label{cor: divbil}
Suppose $n,m>1$, $\pi'=\overline{\pi}$ and $\omega_\pi$ is unitary.
Then any $\Bil_i$ admits a pole in the right half plane $\Re s\ge -1$.
Hence, there exists $\typSh\in\Whit^{\fpsi_{U'}}(\pi)$ such that the integral defining
$\Bil_{\Sh}(\typSh,\overline{\typSh},s)$ diverges for all $s\le -1$.
In particular, $\int_{U'M'^{\fpsi}\bs G}\abs{\typSh(g)}^2\ dg$ diverges.
\end{corollary}

Indeed, by Corollary \ref{cor: polem-1} we have $\ord_Z(s)>0$ for some $s$ with $\Re s\ge m-1$.
Hence, by Corollary \ref{cor: cmprpole} $\ord_{\Bil}(s-m)>0$ for that $s$ (since $n,m>1$).
Therefore, the integral defining $\Bil_{\Sh}(\typSh,\overline{\typSh},s)$ diverges for all $s\le -1$.
In particular,
\[
\int_{U'M'^{\fpsi}\bs G}\abs{\typSh(g)}^2\ dg=
\int_{Q\bs G}\int_{U'\bs\mir}\abs{\typSh(g)}^2\abs{\det g}^{-m}\ dg
\]
diverges.
%which diverges since the integral defining $\Bil_{\Sh}$ diverges at $s=-m$.

In general, we do not know what precisely is the fractional ideal of $\Z[q^{\pm s}]$ generated by
$Z(\typSh,\typSh',\Phi,s)$ where $\typSh\in\Whit^{\fpsi_{U'}}(\sigma)$, $\typSh'\in\Whit^{\fpsi_{U'}}(\sigma')$
and $\Phi\in\swrz(\Mat_{m,nm}(F))$.
If both $\pi$ and $\pi'$ are unitarizable then we expect that this ideal is generated by
$\prod_{i=0}^{m-1}L(s-i,\pi\times\pi')$, i.e., Corollary \ref{cor: unram} is tight in this case.

\begin{example} \label{ex: poleatm12}
Consider $n=m=2$ and $\pi=\abs{\cdot}\times\abs{\cdot}^{-1}\in\Irr_{\gen}\GL_2$.
Then $\pi=\pi^\vee$ and $L(s,\pi\times\pi^\vee)=L(s,\one_{F^*})^2L(s+2,\one_{F^*})L(s-2,\one_{F^*})$.
Therefore, $L(s,\pi\times\pi')L(s-1,\pi\times\pi')$ has a pole at $s=2$.
However, we do not know whether $Z(\cdot,\cdot,\cdot,s)$ is holomorphic at $s=2$, or equivalently
whether $\Bil_0(\cdot,\cdot)$ is well-defined. Recall that $\Sp(\pi,2)$ is not unramified in this case
(cf.~Remark \ref{rem: unr}).
\end{example}

\begin{corollary}
Suppose that $\pi$ is \AT\ and let $\pi'=\pi^\vee$.
Then $\ord_Z(0)=\ord_{\Bil}(-m)+1$. In particular, if $\pi$ is supercuspidal then
$\Bil$ is holomorphic at $s=-m$. Moreover, let
\[
Z^*(\typSh,\typSh',\Phi)=\lim_{s\rightarrow0}(q^s-1)^{\ord_Z(0)}Z(\typSh,\typSh',\Phi,s)
\]
and
\[
\Bil^*_{\Sh}(\typSh,\typSh')=\lim_{s\rightarrow-m}(q^{s+m}-1)^{\ord_{\Bil}(-m)}\Bil_{\Sh}(\typSh,\typSh',s).
\]
Then there exists a constant $c$ such that
\[
Z^*(\typSh,\typSh',\Phi)=c\Phi(0)\int_{Q\bs G}\Bil_{\Sh}^*(\typSh(\cdot g),\typSh'(\cdot g),-m)\ dg
\]
for all $\typSh\in\Whit^{\fpsi_{U'}}(\sigma)$, $\typSh'\in\Whit^{\fpsi_{U'}^{-1}}(\sigma^\vee)$
and $\Phi\in\swrz(\Mat_{m,nm}(F))$.
\end{corollary}

\begin{proof}
By the local functional equation and Corollary \ref{cor: absconv}, $Z^*(\typSh,\typSh',\Phi)=0$ if $\Phi(0)=0$.
%is a bilinear form in $\typSh$, $\typSh'$ times $\Phi(0)$.
Therefore, the argument of Corollary \ref{cor: divbil} (taking $\Phi$ supported near $\eta$, which localizes
$f_{\Phi,\one_{F^*},s}$ near $Q$)
shows that $\ord_Z(0)=\ord_{\Bil}(-m)+1$. Since $\Res_{s=0}f_{\Phi,\one_{F^*},s}=c\Phi(0)$ we get the required relation
from \eqref{eq: Zbileq}.
\end{proof}

We may view
\[
\int_{Q\bs G}\Bil_{\Sh}^*(\typSh(\cdot g),\typSh'(\cdot g),-m)\ dg
\]
as a regularization of
\[
\int_{M'^\fpsi U'\bs G}\typSh(g)\typSh'(g)\ dg.
\]
(Recall that the latter diverges for $\typSh'=\overline{\typSh}$ if $m>1$.)

\begin{comment}
\section{}
We have
\[
\int_{N\cap\mir^*}\Bil_0(\typZe(\cdot u),\typZe')\psi_N(u)^{-1}\ du=\typZe(e)\typZe'(e)
\]
for any $\typZe\in\Whit^{\fpsi_N}(\sigma)$, $\typZe'\in\Whit^{\fpsi_N^{-1}}(\sigma')$ and
\[
\int_{U'}\Bil_{\Sh}(\typSh(\cdot u),\typSh')\psi_{U'}(u)^{-1}\ du=\typSh(e)\typSh'(e).
\]
for any $\typSh\in\Whit^{\fpsi_{U'}}(\sigma)$, $\typSh'\in\Whit^{\fpsi_{U'}^{-1}}(\sigma')$.
\end{comment}

\section{The case $n=m=2$} \label{sec: n=m=2}

Given $\sigma=\Sp(\pi,m)$ it is natural to ask what is the asymptotic behavior of a
function in $\Whit^{\fpsi_{U'}}(\sigma)$ or (what is essentially the same thing) in $\Kir^\psi(\sigma)$.
In the case $n=2$ or if $n=3$ and $m=2$, $P'^{\fpsi}$ is a spherical subgroup of $G$
and the problem can in principle be analyzed by the methods of \cite{MR3764130}.
We will only treat the case where $m=n=2$ and $\pi$ is supercuspidal.
For $n>2$ (excluding the case $n=3$ and $m=2$) $P'^{\fpsi}$ is no longer a spherical subgroup and
the problem seems to be more difficult than
the analogous problem for $\Whit^{\fpsi_N}(\sigma)$. We have little to say about it.

We note that in the case where $n=2$ and $\sigma$ is unramified, an explicit formula for the unramified $\typSh$ was given by
F.~Sato \cite{MR2153954}. This is a special case of a formula of Sakellaridis \cite{MR2260513}.
In general, it would be an interesting problem to obtain such an explicit formula in the unramified case for any $m,n$.
Once again, this goes beyond the scope of \cite{MR3117308}.

For the rest of this section we consider the very special case where $n=m=2$.
Fix an infinite-dimensional $\pi\in\Irr\GL_2$ and $\sigma=\Sp(\pi,2)$.
The transition map $\trns:\Whit^{\fpsi_N}(\sigma)\rightarrow\Whit^{\fpsi_{U'}}(\sigma)$ is
\[
\typZe\mapsto\typSh=\int_F\typZe(\bar u(x)\cdot) \ dx\ \ \text{where }
\bar u(x)=\begin{pmallmatrix}1&&&\\&1&&\\&x&1&\\&&&1\end{pmallmatrix}.
\]

Recall that $\pi^\vee\simeq\pi\omega_\pi^{-1}$.
Let $\pi_s=\pi\abs{\cdot}^s$.
Fix a pairing $\pairing:\pi_{\frac12}\otimes\pi_{-\frac12}\rightarrow\C$
such that $\pairing(\pi_{\frac12}(g)v_1\otimes\pi_{-\frac12}(g)v_2)=\omega_\pi(\det g)\pairing(v_1\otimes v_2)$
for all $g\in\GL_2$.
For any $v\in\pi_{\frac12}\otimes\pi_{-\frac12}$ let $\mc_v:\GL_2\times\GL_2\rightarrow\C$ be the twisted matrix coefficient
$\mc_v(g_1,g_2)=\pairing((\pi_{\frac12}(g_1)\otimes\pi_{-\frac12}(g_2))(v))$.
Thus, $v\mapsto\mc_v$ defines an equivariant map from $\pi_{\frac12}\otimes\pi_{-\frac12}$ to
$\Ind_{(Z\times Z)\GL_2^{\diag}}^{\GL_2\times\GL_2}\chi$
where $Z$ is the center of $\GL_2$, $\GL_2^{\diag}$ is $\GL_2$ diagonally embedded in $\GL_2\times\GL_2$
and $\chi((\lambda_1 I_2,\lambda_2 I_2)(g,g))=\abs{\frac{\lambda_1}{\lambda_2}}\omega_\pi(\lambda_1\lambda_2\det g)$.
If $\pi$ is supercuspidal then the image is contained in $\ind_{(Z\times Z)\GL_2^{\diag}}^{\GL_2\times\GL_2}\chi$.
If $\pi$ is \AT\ then upon identifying
$\pi_{\frac12}\otimes\pi_{-\frac12}$ with $\Whit^{\fpsi_{N_M}}(\pi_{\frac12}\otimes\pi_{-\frac12})$ we may realize $\pairing$ as
\begin{equation} \label{eq: realpair}
\pairing(W)=\int_{F^*}W(\diag(1,-t,1,t))\omega_\pi(t)^{-1}\ d^*t,\ \ W\in\Whit^{\fpsi_{N_M}}(\pi_{\frac12}\otimes\pi_{-\frac12}).
\end{equation}

It follows from the Schur orthogonality relations that if $\pi,\pi'\in\Irr_{\cusp}\GL_2$ with $\omega_\pi\omega_{\pi'}=1$ then
$\pi'$ is equivalent to $\pi^\vee$ if and only if
\begin{equation} \label{eq: orth}
\int_{Z\bs\GL_2}\mc_v(g,1)\mc_{v'}(g,1)\frac{dg}{\abs{\det g}}\ne0
\end{equation}
for some $v\in\pi_{\frac12}\otimes\pi_{-\frac12}$, $v'\in\pi'_{\frac12}\otimes\pi'_{-\frac12}$.

Recall that in the case at hand, $Q=P'=P^w$ where $w=\begin{pmallmatrix}1&&&\\&&1&\\&1&&\\&&&1\end{pmallmatrix}$
and that $\omega_\pi^{\fpsi}$ is the character of $P'^{\fpsi}$
whose restriction to $U'$ is $\fpsi_{U'}$ and whose composition with $\iota$ is $\omega_\pi\circ\det$.
Also, $\norm{\left(\begin{smallmatrix}a&b\\c&d\end{smallmatrix}\right)}=\max(\abs{a},\abs{b},\abs{c},\abs{d})$.

\begin{proposition} \label{prop: structrest}
Suppose that $\pi\in\Irr_{\cusp}\GL_2$. Then we have a short exact sequence of $Q$-modules
\[
0\rightarrow\ind_{P'^{\fpsi}}^Q\omega_\pi^{\fpsi}\rightarrow\sigma\rest_Q\xrightarrow{A}
\pi_{\frac12}\otimes\pi_{-\frac12}\rightarrow0
\]
where $Q$ acts on $\pi_{\frac12}\otimes\pi_{-\frac12}$ through $M'$ (identified with $\GL_2\times\GL_2$
via $\mcp$).
Upon identifying $\sigma\rest_Q$ with $\Kir^{\psi}(\sigma)$, the map $A$ is characterized by the property that
for any $\typKir\in\Kir^{\psi}(\sigma)$ there exists $c>0$ such that
\begin{equation} \label{eq: asympKS}
\typKir(\mcp(g_1,g_2))=\mc_\varphi(g_1,g_2)
\text{ for all $g_1,g_2\in\GL_2(F)$ such that $\norm{g_2^{-1}g_1}\le c$}
\end{equation}
where $\varphi=A(\typKir)$.
Moreover,
\begin{equation} \label{eq: csp}
\text{$\typKir(\mcp(\cdot,1))$ is compactly supported in $\{g\in\GL_2(F):\norm{g}\ge c\}$.}
\end{equation}
\end{proposition}

\begin{proof}
First note that the property \eqref{eq: asympKS} determines $\varphi$ uniquely (if it exists).
It then also follows that if \eqref{eq: asympKS} is satisfied then $A$ necessarily intertwines the $Q$-action.
Moreover, if \eqref{eq: csp} is satisfied then $\varphi=0$ if and only if $\typKir$ is compactly supported modulo $P'^{\fpsi}$.
Also note that in the relation \eqref{eq: asympKS} it is enough to consider $g_2=1$ since both sides are
$(\GL_2^{\diag},\omega_\pi\circ\det)$-equivariant. (For simplicity write $g=g_1$.)

Recall that by Lemma \ref{lem: suppWsh} there exists a constant $C_1>1$ such that $\typKir(\mcp(g,1))=0$ unless $\norm{g}\le C_1$.

Suppose that $\typKir=\trns(\typZe)\rest_Q$.
Write $g=u'(y)\diag(t_1,t_2)k$ where $u'(y)=\begin{pmallmatrix}1&y\\&1\end{pmallmatrix}$ with $k\in\GL_2(\OOO)$.
We claim that there exists $C_3$ such that
\begin{equation} \label{eq: Weq}
\typZe(\bar u(x)\mcp(g,1))=
\abs{x}^{-1}\omega_\pi(x)\typZe(\diag(1,-x^{-1},1,x^{-1})\diag(g,I_2)w)
\end{equation}
for all $x\in F$ such that $\abs{x}>C_3\abs{t_2}$.

Indeed, write
\[
\bar u(x)=u(x^{-1})\diag(1,1,x,x)\diag(1,-x^{-1},1,x^{-1})wu(x^{-1})\ \ \text{where }
u(y)=\begin{pmallmatrix}1&&&\\&1&y&\\&&1&\\&&&1\end{pmallmatrix}.
\]
Then
\[
\bar u(x)\mcp(g,1)=u(x^{-1})\diag(1,1,x,x)\diag(1,-x^{-1},1,x^{-1})w\mcp(g,1)u(t_2x^{-1})^{\mcp(k,1)}
\]
where the superscript denotes conjugation. Our claim follows since $w\mcp(g,1)w^{-1}=\diag(g,I_2)$.

Next, we show that there exists a compact set $C$ of $\GL_2(F)$ such that if $\norm{g}\le C_1$ and $g\notin C$ then
both sides of \eqref{eq: Weq} vanish if $\abs{x}\le C_3\abs{t_2}$.

First note that the condition $\norm{g}\le C_1$ means that $\abs{t_1}, \abs{t_2}, \abs{t_2y}\le C_1$.
Now,
\[
\diag(1,-x^{-1},1,x^{-1})\diag(g,I_2)\in N\diag(t_1,-x^{-1}t_2,1,x^{-1})\GL_4(\OOO).
\]
Therefore, if the right-hand side of \eqref{eq: Weq} is non-zero then by the supercuspidality of $\pi$,
$x$ and $t_1t_2^{-1}$ are confined to a compact subset of $F^*$. Since $\abs{x}\le C_3\abs{t_2}$ we infer
that $t_2$ belongs to a compact set of $F^*$, and hence also $t_1$. Finally, $\abs{y}$ is bounded since $\abs{t_2y}\le C_1$.
Hence, $g$ belongs to a compact set.

On the other hand,
\[
\bar u(x)\mcp(g,1)\in N\diag(t_1,1,t_2,1)\bar u(t_2^{-1}x)K
\]
and since $\abs{t_2^{-1}x}\le C_3$ we infer from the supercuspidality of $\pi$ that if the left-hand side of \eqref{eq: Weq}
is non-zero then $t_1$, $t_2$ belong to a compact subset of $F^*$. As before, $g$ belongs to a compact set. Our claim follows.

In conclusion, \eqref{eq: Weq} holds for all $x\in F$ provided that $\norm{g}\le C_1$ and $g\notin C$.
Integrating \eqref{eq: Weq} over $x\in F$ we conclude that if $\norm{g}\le C_1$ and $g\notin C$ then
\[
\typKir(\mcp(g,1))=(1-q^{-1})\int_{F^*}\omega_\pi(t^{-1})\typZe(\diag(1,-t,1,t)\diag(g,I_2)w)\ d^*t.
\]
By \eqref{eq: realpair} this is equal to $\mc_\varphi(g,1)$
where $\varphi\in\Whit^{\fpsi_{N_M}}(\pi_{\frac12}\otimes\pi_{-\frac12})$ is $(1-q^{-1})$
times the restriction of $\typZe(\cdot w)$ to $M$.
Thus, \eqref{eq: asympKS} and \eqref{eq: csp} hold.
In view of Corollary \ref{cor: compind} this proves the proposition.
(Note that $\typZe\mapsto\varphi$ is $Q$-equivariant since $w$ conjugates $P$ to $Q$.)
\end{proof}

\begin{remark}
%We can easily find $0\ne\varphi\in\sigma\subset\pi\abs{\det}^{-\frac12}\times\pi\abs{\det}^{\frac12}$
%such that $0=\varphi(e)\in\pi\otimes\pi$. Indeed, $\sigma$ is the image of the intertwining operator
%\[
%M\varphi(g)=\int_U\varphi(wug)\ du.
%\]
%For any $X\in\Mat_{2,2}(F)$ let $u(X)=\begin{pmatrix}I_2&X\\&I_2\end{pmatrix}\in U$.
%Take $\varphi$ supported in $PwU$ such that $\varphi(wu(X))=\phi(X)v$ for some $\phi\in\swrz(\Mat_{2,2}(F))$ and $v\in\pi\otimes\pi^\vee$.
%Then $M\varphi(e)=(\int_{\Mat_{2,2}(F)}\phi(X)\ dX)v$. On the other hand,
%\begin{multline*}
%M\varphi(w)=\int_{\bar U}\varphi(\bar u)\ d\bar u=
%\int_{\Mat_{2,2}(F)}\varphi(\begin{pmatrix}Y^{-1}&*\\&Y\end{pmatrix}wu(Y^{-1}))\ dY
%\\=\int_{\Mat_{2,2}(F)}(\pi(Y^{-1})\otimes\pi(Y))(v)\abs{\det Y}^{-3}\phi(Y^{-1})\ dY\\=
%\int_{\Mat_{2,2}(F)}(\pi(Y)\otimes\pi(Y^{-1}))(v)\abs{\det Y}^{-1}\phi(Y)\ dY.
%\end{multline*}
%Take an arbitrary $0\ne v\in\pi\otimes\pi$ and let $Y_0\in\GL_2(F)$ be such that $(\pi(Y_0)\otimes\pi(Y_0^{-1}))v$
%is not proportional to $v$. By taking $\phi$ supported in the union of small neighborhoods of the identity
%and $Y_0$ we can arrange that $M\varphi(e)=0$ while $M\varphi(w)\ne0$.
%Note that $\typSh\rest_{\mir}$ is compactly supported modulo $U'$ if and only if $\varphi=0$ in the notation
%of Lemma \ref{lem: asymp}.
It follows from (the proof of) Proposition \ref{prop: structrest} that there exists a non-zero $\typZe\in\Whit^{\fpsi_N}(\sigma)$
that vanishes on $M$ (in which case $\trns\typZe(\cdot w)\rest_Q\in\Kir^{\psi}(\sigma)$
is compactly supported modulo $P'^{\fpsi}$).
This can be also shown directly by realizing $\sigma$ as the image of the intertwining operator
\[
\pi_{\frac12}\times\pi_{-\frac12}\rightarrow\pi_{-\frac12}\times\pi_{\frac12}
\]
and taking the image of a suitable vector in $\pi_{\frac12}\times\pi_{-\frac12}$ that is supported in the big cell.
\end{remark}

\begin{corollary}
Suppose that $\pi\in\Irr_{\cusp}\GL_2$ and let $\pi'=\pi^\vee$.
Then the poles of the bilinear form $\Bil_i(\typSh,\typSh',s)$, as we vary $\typSh\in\Whit^{\fpsi_{U'}}(\sigma)$
and $\typSh'\in\Whit^{\fpsi_{U'}^{-1}}(\sigma')$, coincide with those of $L(s+1,\pi\times\pi^\vee)$.
\end{corollary}

\begin{proof}
We may assume without loss of generality that $\pi$ is unitary. Then
\[
\Bil(\typSh,\overline{\typSh},s-1)=\int_{\GL_2}\abs{\det g}^{s-1}\abs{\typSh(\mcp(g,1))}^2\ dg.
\]
By Proposition \ref{prop: structrest}, the analytic properties are governed by those of
\[
\int_{\GL_2:\norm{g}\le1}\abs{\det g}^{s-1}\abs{\mc_\varphi(g,1)}^2\ dg
\]
which can be written as
\begin{multline*}
\int_{Z\bs\GL_2}(\int_{F^*:\abs{\lambda}\le\norm{g}^{-1}}\abs{\lambda}^{2s}\ d\lambda)\abs{\mc_\varphi(g,1)}^2\abs{\det g}^{s-1}\ dg
\\=(1-q^{-2s})^{-1}\int_{Z\bs\GL_2}\norm{g}^{-2s}\abs{\mc_\varphi(g,1)}^2\abs{\det g}^{s-1}\ dg.
\end{multline*}
Thus, the poles are simple and are confined to $q^{2s}=1$.
If $q^s=1$ then the residue is clearly non-zero. If $q^s=-1$ then the residue is a constant multiple of
\[
\int_{Z\bs\GL_2}\abs{\mc_\varphi(g,1)}^2\omega(\det g)\abs{\det g}^{-1}\ dg
\]
where $\omega$ is the non-trivial quadratic unramified character of $F^*$.
Thus, by \eqref{eq: orth} the residue is non-zero if and only if $\pi\simeq\pi\omega$.
This matches exactly with the poles of $L(s,\pi\times\pi^\vee)$ (\cite[Proposition 8.1]{MR701565}).
\end{proof}

\begin{comment}
\subsection{}
\Erez{Not clear whether should be included. Just wrote it in order to remember..}
Consider $\pi=\overbrace{\one_{F^*}\times\dots\times\one_{F^*}}^n\in\Irr_{\temp}\GL_n$
so that $\sigma=\Ind_{P'}^G\one_{M'}$.
Let $\typSh$ be the unramified vector in $\Whit^{\fpsi_{U'}}(\sigma)$ such that $\typSh(e)=1$.
Consider the restriction of $\typSh$ to $\mir\cap M'$.
Let $\varphi$ be the unramified element in $\Ind_{P'}^G1$ with $\varphi(e)=1$.
Then for any $l_1,\dots,l_{n-1}\in\GL_m$
\[
\typSh(\diag(l_1,\dots,l_{n-1},1))
\]
is $\delta_{P'}(l)^{\frac12}$ times the Fourier transform $\hat f$ of $\varphi(w\cdot)\rest_{U'}$ at
$l_2^{-1}l_1,l_3^{-1}l_2,\dots,l_{n-1}$.
Note that $\delta_{P'}\rest_{\mir\cap P'}=\delta_{P'\cap\mir}\abs{\det}^m$.
Therefore
\[
\Bil_{\Sh}(\typSh,\overline{\typSh})=
\int_{\GL_m^{n-1}}\abs{\hat f(l_2^{-1}l_1,\dots,l_{n-1})}^2\abs{\det l_1\dots l_{n-1}}^m\ dl_1\dots dl_{n-1}.
\]
Making change of variables $l_1\mapsto l_2l_1$, $l_2\mapsto l_3l_2$, \dots $l_{n-2}\mapsto l_{n-1}l_{n-2}$ we get
\begin{multline*}
\int_{\GL_m^{n-1}}\abs{\hat f(l_1,\dots,l_{n-1})}^2\prod_{i=1}^{n-1}\abs{\det l_i}^{mi}\ dl_1\dots\ dl_{n-1}\\=
\int_{\Mat_{m,m}^{n-1}}\abs{\hat f(X_1,\dots,X_{n-1})}^2\prod_{i=1}^{n-1}\abs{\det X_i}^{m(i-1)}\ dX_1\dots\ dX_{n-1}.
\end{multline*}
\begin{question}
Does the integral
\[
\int_{\Mat_{m,m}^{n-1}}\abs{\hat f(X_1,\dots,X_{n-1})}^2\prod_{i=1}^{n-1}\ dX_1\dots\ dX_{n-1}
\]
converge?
\end{question}
\end{comment}

\section{Global heuristics} \label{sec: global}

Let $F$ be a number field with ring of adeles $\A$.
We consider $G=\GL_{nm}$ as a group over $F$ and write $G(\A)^1=\{g\in G(\A):\abs{\det g}=1\}$.
As before, let $Q$ be the stabilizer of $\spn\{e_{ni}:i=1,\dots,m\}$ in $G$ -- a maximal non-standard parabolic subgroup of $G$
of type $((n-1)m,m)$.
For any $\Phi\in\swrz(M_{m,nm}(\A))$ and a Hecke character $\omega$ of $F^*\bs\A^*$
consider the degenerate normalized Eisenstein series that is given by
\begin{multline*}
\Eisen_{\Phi,\omega}(g,s)=\int_{\GL_m(F)\bs\GL_m(\A)}\sum_{\gamma\in M_{m,mn}(F):\rk\gamma=m}\Phi(z^{-1}\gamma g)
\abs{\det z}^{-ns}\omega(\det z)^{-1}\abs{\det g}^s\ dz\\=
\sum_{\gamma\in Q\bs G}f_{\Phi,\omega,s}(\gamma g)
\end{multline*}
for $\Re s\gg0$ (more precisely, $\Re s>m$ if $\omega$ is unitary) where as in \S\ref{sec: fphi}
\[
f_{\Phi,\omega,s}(g)=\int_{\GL_m(\A)}\Phi'_g(l)\omega(\det l)\abs{\det l}^{ns}\ dl\ \abs{\det g}^s
\]
and $\Phi'_g\in\swrz(\Mat_{m,m}(\A))$ is given by $\Phi'_g(X)=\Phi(\mu(X)g)$ where $\mu(X)\in\Mat_{m\times nm}(\A)$
is the matrix whose $i$-th row is $\sum_{j=1}^mX_{i,j}e_{nj}$.
By the method of Tate's thesis (which goes back to Riemann) $\Eisen_{\Phi,\omega}$ admits a meromorphic continuation
with finitely many (simple) poles and a functional equation
\[
\Eisen_{\Phi,\omega}(g,s)=\Eisen_{\hat\Phi,\omega^{-1}}(\,^tg^{-1},m-s).
\]

As before, let $P=M\ltimes U$ be the standard maximal parabolic subgroup of $G$ of type $(\overbrace{n,\dots,n}^m)$
and let $\abs{\cdot}_M:M(\A)\rightarrow \R_{>0}^m$ be the homomorphism
\[
\abs{\diag(l_1,\dots,l_m)}=(\abs{\det l_1},\dots,\abs{\det l_m}).
\]
We extend $\abs{\cdot}_M$ to a left $U(\A)$ and right $K$-invariant function $\abs{\cdot}_P$ on $G(\A)$
where $K$ is the standard maximal compact subgroup of $G(\A)$.
For any $x=(x_1,\dots,x_m)\in\R_{>0}^m$ and $\lambda=(\lambda_1,\dots,\lambda_m)\in\C^m$ we write
$x^\lambda=\prod_ix_i^{\lambda_i}$.

Let $\pi=\otimes\pi_v$ be an irreducible cuspidal representations of $\GL_n(\A)$.
Let $\phi:G(\A)\rightarrow\C$ be a smooth function such that for all $g\in G(\A)$
the function $l\in M(\A)\mapsto\delta_P(l)^{-\frac12}\phi(l g)$ belongs to the space
of $\overbrace{\pi\otimes\dots\otimes\pi}^m$. The Eisenstein series
\[
E(\phi,\lambda,g)=\sum_{\gamma\in P(F)\bs G(F)}\phi(\gamma g)\abs{\gamma g}_P^\lambda
\]
converges if $\Re(\lambda_i-\lambda_{i+1})>n$ for all $i=1,\dots,m-1$ and admits a meromorphic continuation to $\C^n$.
The limit
\[
\varphi(g)=\lim_{\lambda\rightarrow(\frac{m-1}2,\dots,\frac{1-m}2)}(\lambda_1-\lambda_2-1)\dots
(\lambda_{m-1}-\lambda_m-1)E(\phi,\lambda,g)
\]
exists and is a square-integrable automorphic form on $G(F)\bs G(\A)^1$ which is non-zero for a suitable $\phi$.
As we vary $\phi$, we obtain an irreducible automorphic representation of $G(\A)$ whose local components
are $\Sp(\pi_v,m)$.
Similarly, let $\pi'$ be another irreducible cuspidal representation of $\GL_n(\A)$ and let $\phi'$ and $\varphi'$
be analogous functions with respect to $\pi'$.

Formally, we would have liked to consider the integral
\begin{equation} \label{eq: globint}
\int_{G(F)\bs G(\A)^1}\varphi(g)\varphi'(g)\Eisen_{\Phi,\omega}(g,s)\ dg
\end{equation}
where $\omega=\omega_{\pi}\omega_{\pi'}$.
For $m=1$, this is of course the classical Rankin--Selberg integral. Unfortunately, for $m>1$ this integral does not converge
as none of the functions that appear in the integrand is rapidly decreasing.
A suitable regularization is therefore needed in order to make sense of \eqref{eq: globint}.
We will not pursue this matter here. Instead, we will be content with a \emph{purely heuristic} argument,
anticipating what a possible regularization of \eqref{eq: globint} would yield.

\begin{comment}
Unfortunately the integral doesn't converge for any $s$.
At $s=0$ the worst exponent of $E_\Phi$ is $\lambda-\rho$ where
\[
\lambda=(\overbrace{\frac{m(n-1)}2}^m,\overbrace{-\frac m2}^{(n-1)m})
\]
Its projection to the parabolic of type $(\overbrace{n,\dots,n}^m)$ is (at least if $n\ge m$)
\[
\mu_1=({m\choose 2},\overbrace{-\frac m2}^{m-1})
\]
On the other hand, the exponent of the Speh representation is
$\mu_2=(-\frac{m-1}2,\dots,\frac{m-1}2)$.
We would need $2\mu_2+\mu_1$ to be negative which is not the case.
In the case $m=2$ it is $0$ which would suggest that we could replace $E_\Phi$ by the pseudo-Eisenstein series.
However, this is no longer the case for $m>2$. \Erez{Can we carry out the computation with a pseudo-Eisenstein series
instead of one of the Speh representations? What would that mean?}
\end{comment}

As in the case $m=1$ we unfold (formally) the expression \eqref{eq: globint}.
For any $i=1,\dots,m$ let $Q_i=L_i\ltimes V_i$ be the stabilizer of the flag
\[
(\spn\{e_{nj-k}:j=1,\dots,m,\ k=0,\dots,r-1\})_{r=1,\dots,i}
\]
in $G$. Thus, $Q_1=Q\supset Q_2\supset\dots\supset Q_{n-1}=Q_n=P'$ and $L_i\simeq\GL_{m(n-i)}\times L_i'$ with $L_i'\simeq\overbrace{\GL_m\times\dots\times\GL_m}^i$.
Let $p_i:Q_i\rightarrow L_i'$ be the resulting projection
and let $Q_i'$ be the inverse image of $\GL_m$ diagonally embedded in $L_i'$.
In particular, $Q_1'=Q_1=Q$ and $Q_n'=M'^{\fpsi}\ltimes U'$.
Note that for all $i=1,\dots,n-1$, $Q_{i+1}'$ is the stabilizer in $Q_i$ of the character $\fpsi_{V_i}$
and $V_i/V_{i-1}$ is abelian (and can be identified with $\Mat_{m,(n-i)m}$) where for consistency we let $V_0=0$.

In the first step we unfold \eqref{eq: globint} to write it as
\[
\int_{Q(F)\bs G(\A)^1}\varphi(g)\varphi'(g)f_{\Phi,\omega,s}(g)\ dg=
\int_{Q_1(F)\bs G(\A)^1}\int_{V_1(F)\bs V_1(\A)}\varphi(vg)\varphi'(vg)\ dv\ f_{\Phi,\omega,s}(g)\ dg
\]
and expand
\[
\int_{V_1(F)\bs V_1(\A)}\varphi(vg)\varphi'(vg)\ dv=
\sum_{\chi\in\PD(V_1(F)\bs V_1(\A))}\varphi^{V_1,\chi}(g)\varphi'^{V_1,\chi^{-1}}(g)
\]
where
\[
\varphi^{V_1,\chi}(g)=\int_{V_1(F)\bs V_1(\A)}\varphi(vg)\chi(v)^{-1}\ dv.
\]
The Pontryagin dual of the compact abelian group $V_1(F)\bs V_1(\A)$ is isomorphic to
$\Mat_{m,(n-1)m}(F)$.
We consider only the contribution from the non-degenerate $\chi$'s, i.e. those corresponding
to matrices of rank $m$ (anticipating that the degenerate ones will not contribute, either by the cuspidality of $\pi$
or by the regularization procedure itself).
The non-degenerate characters form a single orbit under $Q=Q_1$, namely the orbit of $\fpsi_{V_1}$,
and the stabilizer of $\fpsi_{V_1}$ is $Q_2'$.
We thus get
\[
\int_{Q_2'(F)\bs G(\A)^1}\varphi^{V_1,\fpsi_{V_1}}(g)\varphi'^{V_1,\fpsi_{V_1}^{-1}}(g)f_{\Phi,\omega,s}(g)\ dg
\]
which we write as
\[
\int_{Q_2'(F)\bs G(\A)^1}\int_{V_2(F)\bs V_2(\A)}\varphi^{V_1,\fpsi_{V_1}}(ug)\varphi'^{V_1,\fpsi_{V_1}^{-1}}(ug)\ du
\ f_{\Phi,\omega,s}(g)\ dg.
\]
Once again, we expand the inner integral according to characters of the compact abelian group $V_2(\A)/V_1(\A)V_2(F)$
and consider only the non-degenerate characters.
Continuing this way we get for $k=1,\dots,n$
\[
\int_{Q_k'(F)\bs G(\A)^1}\varphi^{V_{k-1},\fpsi_{V_{k-1}}}(g){\varphi'}^{V_{k-1},\fpsi_{V_{k-1}}^{-1}}(g)f_{\Phi,\omega,s}(g)\ dg.
\]
For $k=n$ we obtain
\[
\int_{M'^{\fpsi}(F)U'(\A)\bs G(\A)^1}\varphi^{U',\fpsi_{U'}}(g)\varphi'^{U',\fpsi_{U'}^{-1}}(g)f_{\Phi,\omega,s}(g)\ dg.
\]
Now, $\varphi^{U',\fpsi_{U'}}$ is $(M'^{\fpsi}(\A),\omega_\pi\circ\det)$-equivariant (taking into account the identification
$\iota:\GL_m\rightarrow M'^{\fpsi}$). Therefore, up to a volume factor we get
\begin{equation} \label{eq: finexp}
\int_{M'^{\fpsi}(\A)U'(\A)\bs G(\A)}\varphi^{U',\fpsi_{U'}}(g){\varphi'}^{U',\fpsi_{U'}^{-1}}(g)f_{\Phi,\omega,s}(g)\ dg.
\end{equation}
This integral (which actually converges for $\Re s>m$ if $\omega$ is unitary) is Eulerian.
%More precisely, assume that $\varphi$ is a factorizable vector in $\pi$.
%We may write
%\[
%\varphi^{U',\fpsi_{U'}}(g)=\prod_v(\typSh)_v(g_v),\ \ (\typSh)_v\in\Whit^{\fpsi_{U'}}(\pi_v).
%\]
%Similarly assume that $\varphi'$ is a factorizable vector in $\pi'$ and write
%\[
%\ \varphi'^{U',\fpsi_{U'}^{-1}}(g)=\prod_v(\typSh')_v(g_v)\ \ (\typSh')_v\in\Whit^{\fpsi_{U'}^{-1}}(\pi_v).
%\]
%Finally assume that $\Phi(X)=\prod_v\Phi_v(X_v)$, $X=(X_v)\in\Mat_{m,nm}(\A)$.
Let $S$ be a finite set of places of $F$ containing all the archimedean ones such that for all $v\notin S$
$\varphi$ and $\varphi'$ are $G(\OOO_v)$-invariant (and in particular, $\pi_v$ and $\pi'_v$ are unramified)
$\psi_v$ has conductor $\OOO_v$, $\Phi$ is invariant under translation by $\Mat_{m,mn}(\OOO_v)$
and $\Phi(X)=0$ unless $X_v\in\Mat_{m,mn}(\OOO_v)$.
Using \eqref{eq: Zbileq} and Corollary \ref{cor: unram} the integral \eqref{eq: finexp} is equal to
\[
\big(\prod_{i=0}^{m-1}L^S(s-i,\pi\times\pi')\big)Z_S(\varphi^{U',\fpsi_{U'}}\rest_{G(F_S)},
{\varphi'}^{U',\fpsi_{U'}^{-1}}\rest_{G(F_S)},\Phi\rest_{\Mat_{m,mn}(F_S)},s)
\]
where $L^S(s,\pi\times\pi')$ is the partial Rankin-Selberg $L$-function and
for any $\typSh\in\Whit^{\fpsi_{U'}}(\Sp(\pi_S,m))$ and $\typSh'\in\Whit^{\fpsi_{U'}^{-1}}(\Sp(\pi'_S,m))$
\[
Z_S(\typSh,\typSh',\Psi,s)=\int_{U'(F_S)\bs G'(F_S)}\typSh(g)\typSh'(g)\Phi(\eta g)\abs{\det g}^s\ dg,
\]
which is essentially the product over $v\in S$ of the integrals considered in \S\ref{sec: locint}.
(We tacitly assume that the analysis of sections \ref{sec: models}--\ref{sec: locint} carries over to the archimedean case.)

\appendix
\section{Relation to intertwining operators}  \label{sec: apend}
For this appendix assume that $\pi\in\Irr_{\AT}\GL_n$.
Let
\[
\Pi=\pi\abs{\cdot}^{\frac{m-1}2}\times\pi\abs{\cdot}^{\frac{m-3}2}\times\dots\times\pi\abs{\cdot}^{\frac{1-m}2}
\]
be the standard module which admits $\sigma=\Sp(\pi,m)$ as the Langlands quotient.
We realize $\Pi$ in the subspace $\Whit^{\fpsi_N}(\Pi)$ of $\Ind_N^G\fpsi_N$
consisting of functions $W$ such that $l\in M\mapsto\delta_P^{-\frac12}(l)\delta'^{-1}(l)W(lg)\in\Whit^{\fpsi_{N_M}}(\pi^{\otimes m})$
for all $g\in G$.
Define an intertwining operator on $\Whit^{\fpsi_N}(\Pi)$ by
\begin{equation}
\label{eq: defM}
W\mapsto MW(\cdot)=\int_{U}W(\ww u\cdot)\ du
\end{equation}
where $\ww$ is as in Remark \ref{rem: w0}.
The integral defining $MW$ is absolutely convergent and its image is $\Whit^{\fpsi_N}(\sigma)$.
Similarly, define $\Whit^{\fpsi_N^{-1}}(\Pi^\vee)\simeq\Pi^\vee$ to be the subspace of $\Ind_N^G\fpsi_N^{-1}$
consisting of functions $W^\vee$ such that
$l\in M\mapsto\delta_P^{-\frac12}(l)\delta'(l)W^\vee(lg)\in\Whit^{\fpsi_{N_M}^{-1}}((\pi^\vee)^{\otimes m})$
for all $g\in G$. Then the bilinear form
\[
\sprod{W}{W^\vee}=\int_{P\bs G}\int_{N_M\bs\mir_M}\delta_P(l)^{-1}W(lg)W^\vee(lg)\ dl\ dg,
\ \ \ W\in \Whit^{\fpsi_N}(\Pi),\ W^\vee\in \Whit^{\fpsi_N^{-1}}(\Pi^\vee)
\]
converges absolutely and defines a $G$-invariant pairing on $\Whit^{\fpsi_N}(\Pi)\times\Whit^{\fpsi_N^{-1}}(\Pi^\vee)$
where $\mir_M=\mir\cap M$ is the product of $m$ copies of the mirabolic subgroup of $\GL_n$.
Since $\Whit^{\fpsi_N^{-1}}(\sigma^\vee)$ is the socle of $\Whit^{\fpsi_N^{-1}}(\Pi^\vee)$, for any
$\typZe^\vee\in \Whit^{\fpsi_N^{-1}}(\sigma^\vee)$ the linear form $W\mapsto\sprod{W}{\typZe^\vee}$ factors through $MW$
and it is a scalar multiple (independently of $W$) of $\Bil_0(MW,\typZe^\vee)$.
In the rest of the appendix we prove the following identity.

\begin{proposition} \label{prop: appmain}
For any $W\in\Whit^{\fpsi_N}(\Pi)$ and $\typZe^\vee\in\Whit^{\fpsi_N^{-1}}(\sigma^\vee)$ we have
\begin{equation}\label{eq: innercomp}
\sprod{W}{\typZe^\vee}=\omega_\pi(-1)^{m\choose2}\Bil_0(MW,\typZe^\vee).
\end{equation}
\end{proposition}

The identity will follow from a series of identities proved below.

For $i=1,\ldots,m-1$, let $U^i$ be the unipotent radical of the standard parabolic subgroup $P^i$ of $G$ of type $(in,n,n,\ldots,n)$.
Let $\bar U^i=\,^tU^i$ be its opposite.
\begin{lemma}\label{L: A1}
% The function $\typZe$ is compactly supported on $\mir\cap \bar U$ and we have \Mao{selfdual measure}
Let $\tau\in \Irr_{\gen}\GL_n$. Then for any $\typZe\in \Whit^{\fpsi_N}(\Sp(\tau,m))$ we have
\begin{equation}\label{eq: A1}
\int_{\mir\cap \bar U^i} \typZe(\bar u )\ d\bar u=\omega_\tau(\det\wi) \typZe(\iota(\wi))
\end{equation}
where the integrand on the left-hand side is compactly supported.
Here $\wi=\sm{}{w_{m-i}}{I_i}{}$ where $w_{m-i}$ is the longest Weyl element in $\GL_{m-i}$.
Thus,
\[
\int_{\mir\cap \bar U^i} \typZe(\bar u\bar v)\ d\bar u=\int_{\mir\cap \bar U^i} \typZe(\bar u )\ d\bar u
\]
for all $\bar v\in\bar U^i$.
In particular for $i=1$
\[
\int_{\bar U_\mir} \typZe(\bar u )\ d\bar u=\omega_\tau(-1)^{m\choose2} \typZe(\ww)
\]
where $\bar U_\mir=\bar U\cap\mir$.
\end{lemma}

\begin{proof}
Let $\typSh^i=\trns_{(i-1)n}\typZe\in\Whit^{\fpsi_{U_{(i-1)n}}}(\Sp(\tau,m))$.
Recall that $U_{(i-1)n}$ is the subgroup of $P^i$ consisting of matrices whose $n\times n$ blocks $A_{j,k}$ satisfy
\begin{itemize}
\item $A_{j,j}$ is upper unitriangular for all $j=1,\dots,m$,
\item $A_{j,k}$ is strictly upper triangular if $j\neq k$ and $j,k\leq i$,
\item $A_{j,k}=0$ if $j>k$ and $j>i$.
\end{itemize}
(There are no conditions on $A_{j,k}$ if $k>j$ and $k>i$.)

%the mixed model that lies in $\Ind \typSh(\Sp(\pi,m-1)\abs{\cdot}^{-\frac12})\otimes W(\pi\abs{\cdot}^{\frac{m-1}{2}})$.
The inverse transform in Proposition~\ref{prop: trans2} gives
\[
\typZe(g)=\int_{N\cap U_{(i-1)n}\bs N_\mir} \typSh^i(u g )\ du.
\]
We may replace the domain of integration by $( N\cap U_{(i-1)n}\cap\GL_{in})\bs ( N_\mir\cap\GL_{in}) $
where $\GL_{in}$ is embedded in $G$ by $h\mapsto \sm{h}{}{}{I_{(m-i)n}}$.
Let $U_i'= U_{(i-1)n}\cap\GL_{in}=U'\cap \GL_{in}$ and $\mir_i=\mir\cap\GL_{in}$.
Thus, the above integral can be taken over $N\cap U'_i\bs N\cap\mir_i $, and by Lemma~\ref{L: support2} the integrand is compactly supported.

The expression on the left-hand side of \eqref{eq: A1} is
\[
\int_{\mir\cap \bar U^i} \int_{N\cap U'_i\bs N\cap\mir_i } \typSh^i(u \bar u )\ du\ d\bar u.
\]
The same argument as in Lemma~\ref{L: support} shows the function $\typSh^i(u \bar u)$ is compactly supported in $\bar u$ uniformly in $u$.
Thus, the above double integral is absolutely convergent.
Changing the order of integration and making a change of variable in $\bar u$ we get
\[
\int_{N\cap U'_i\bs N\cap\mir_i } \int_{\mir\cap \bar U^i}  \typSh^i( \bar u u )\ d\bar u\ du.
\]
Notice that the partial integration over $U'\cap \bar U^i\subset \mir\cap \bar U^i$ is the composition of the transforms $\trns_j$ defined in
Proposition~\ref{prop: trans2} for $j=(i-1)n,\ldots,(m-1)n-1$. Thus, the above is
\[
\int_{N\cap U'_i\bs N\cap\mir_i } \int_{U'\cap \bar U^i\bs\mir\cap \bar U^i}  \typSh( \bar u u )\ d\bar u\ du
\]
where $\typSh=\trns\typZe$.
By Lemma \ref{lem: equicchar} $\typSh(\iota(\wi) g)=\omega_\tau(\det(\wi))\typSh(g)$. The above becomes
\begin{multline*}
\omega_\tau(\det(\wi))\int_{N\cap U'_i\bs N\cap\mir_i } \int_{U'\cap \bar U^i\bs\mir\cap \bar U^i}
 \typSh( \iota(\wi)\bar u u )\ d\bar u\ du
\\=\omega_\tau(\det(\wi))\int_{N\cap U'\bs N_\mir}   \typSh( u \iota(\wi))\ du.\end{multline*}
Now Lemma \ref{L: support2} gives \eqref{eq: A1}.
For the second part, we only need to note that for all $\bar v\in \bar U^i$
we have $\typZe(\iota(\wi)\bar v)=\typZe(\iota(\wi))$.
\end{proof}

Write $\bar U$ as a (semidirect) product of abelian groups $\bar U_2\bar U_3\ldots \bar U_m$, where $\bar U_i$ consists of
the elements $\bar u$ in $\bar U$ such that $\bar u_{j,k}=\delta_{j,k}$ if $j\leq n(i-1)$ or $j>ni$.
For brevity, for any $i=1,\dots,m$ we denote the iterated integral
\[
\int_{\bar U_2\cap\mir}\big(\int_{\bar U_3\cap\mir}\cdots\big(\int_{\bar U_i\cap\mir}
f(\bar u_i\cdots\bar u_3\bar u_2)\ d\bar u_i\big)\cdots\ d\bar u_3\big)\ d\bar u_2
\]
(assuming convergent) by
\[
\itint_{\bar U\cap\mir_i}f(\bar u)\ du.
\]

\begin{lemma} \label{L: A2}
Let $\typZe$ be as before and let $\phi\in\swrz(\bar U)$.
Then
\[
\itint_{\bar U_\mir} \int_{\bar U}\phi(\bar u)\typZe(\bar u \bar v)\ d\bar u \ d\bar v=
\omega_\tau(-1)^{m\choose 2} \typZe(\ww)\int_{\bar U}\phi(\bar u)\ du
\]
where each integrand in the iterated integral on the left-hand side is compactly supported.
\end{lemma}

\begin{proof}
We show by descending induction on $i=1,\dots,m$ that the left-hand side is equal to
\begin{equation} \label{eq: istep}
\itint_{\bar U\cap\mir_i}\int_{\bar U^i\cap\mir}\int_{\bar U}\phi(\bar u)\typZe(\bar v\bar u\bar v')\ d\bar u
\ d\bar v\ d\bar v'.
\end{equation}
Note that the integrand is compactly supported in $\bar u$ and $\bar v$.
For $i=m$ this is clear while for $i=1$ we obtain the statement of the lemma by Lemma \ref{L: A1}.

For the induction step, we assume $i>1$ and use Lemma \ref{L: A1} to rewrite \eqref{eq: istep} as
\[
\itint_{\bar U\cap\mir_i}\int_{\bar U^i\cap\mir}\int_{\bar U}\phi(\bar u)\typZe(\bar vp_i(\bar u)\bar v')\ d\bar u
\ d\bar v\ d\bar v'
\]
where $p_i:\GL_{in}\ltimes\bar U^i\rightarrow\GL_{in}$ is the projection.
Now write $\bar v'=\bar v_1\bar v_2$ where $\bar v_1\in\bar U_i\cap\mir$ and $\bar v_2\in\bar U\cap\mir_{i-1}$
and note that $p_i(\bar u)\in\GL_{in}\cap\bar U$ normalizes $\bar U_i\cap\mir$. Therefore,  \eqref{eq: istep} is equal to
\begin{multline*}
\itint_{\bar U\cap\mir_{i-1}}\int_{\bar U_i\cap\mir}\int_{\bar U^i\cap\mir}\int_{\bar U}
\phi(\bar u)\typZe(\bar v\bar v_1p_i(\bar u)\bar v_2)\ d\bar u\ d\bar v\ d\bar v_1\ d\bar v_2
\\=\itint_{\bar U\cap\mir_{i-1}}\int_{\bar U^{i-1}\cap\mir}\int_{\bar U}
\phi(\bar u)\typZe(\bar vp_i(\bar u)\bar v')\ d\bar u\ d\bar v\ d\bar v'
\\=\itint_{\bar U\cap\mir_{i-1}}\int_{\bar U^{i-1}\cap\mir}\int_{\bar U}
\phi(\bar u)\typZe(\bar v\bar u\bar v')\ d\bar u\ d\bar v\ d\bar v'
\end{multline*}
as required.
\end{proof}

Denote by $\delta_{\bar U_\mir}(l)$ the character of $\mir_M$ given by $d(l\bar u l^{-1})=\delta_{\bar U_\mir}(l)\ d\bar u$
where $d\bar u$ is a Haar measure on $\bar U_\mir$.

Let $\Whit^{\fpsi_N}(\Pi)_\sharp$ be the linear subspace of $\Whit^{\fpsi_N}(\Pi)$ generated by the functions $W$ of the form
\[
W(g)=\begin{cases}\delta_P(l)^{\frac12}\delta'(l)W'(l)\phi(\bar u)&\text{if }g=ul\bar u,\ u\in U,\ l\in M,\ \bar u\in\bar U,\\
0&\text{otherwise,}\end{cases}
\]
where $\phi\in\swrz(\bar U)$, $W'\in\Whit^{\fpsi_{N_M}}(\pi^{\otimes m})$ and $W'\rest_{\mir_M}$ is compactly supported modulo $N_M$.

\begin{lemma} \label{L: A3}
For any $W\in \Whit^{\fpsi_N}(\Pi)_\sharp$ and $\typZe^\vee\in\Whit^{\fpsi_N^{-1}}(\sigma^\vee)$ we have
\[
\itint_{\bar U_\mir} \sprod{W}{\typZe^\vee(\cdot \bar v)} \ d\bar v= \omega_\pi(-1)^{m\choose2}
\int_{N_M\bs\mir_M} MW(\ww l) \typZe^\vee(\ww l)\delta_{\bar U_\mir}^{-1}(l)\ dl.
\]
\end{lemma}

\begin{proof}
The left-hand side is
\begin{multline*}
\itint_{\bar U_\mir}\int_{\bar U}\int_{N_M\bs\mir_M}\delta_P(l)^{-1}W(l\bar u)\typZe^\vee(l\bar u\bar v)\ dl\ d\bar u\ d\bar v\\=
\itint_{\bar U_\mir}\int_{N_M\bs\mir_M}\int_{\bar U}W(\bar ul)\typZe^\vee(\bar ul\bar v)\ d\bar u\ dl\ d\bar v.
\end{multline*}
%The integral is absolutely convergent by our choice of $W$ and Lemma~\ref{L: A2}.
As $M$ normalizes $\bar U_i$ for any $i$, this equals
\[
\int_{N_M\bs\mir_M}\itint_{\bar U_\mir}\int_{\bar U} W(\bar u l)\typZe^\vee(\bar u\bar v l)
\delta_{\bar U_\mir}^{-1}(l)\ d\bar u\ d\bar v\ dl.
\]
Here we can interchange the order of integration as $l$ is integrated over a fixed compact set.
The claim now follows from Lemma~\ref{L: A2} and \eqref{eq: defM}.
\end{proof}

Let $\Whit^{\fpsi_N^{-1}}(\sigma^\vee)_\flat$ be the subspace of $\Whit^{\fpsi_N^{-1}}(\sigma^\vee)$
consisting of the functions $\typZe^\vee$ such that $\typZe^\vee\rest_\mir$
is contained in $\ind_{N_\mir}^\mir\fpsi_N$ and $\typZe^\vee\rest_\mir$ is supported in $P\bar U\cap\mir$.

\begin{lemma} \label{L: A4}
For any $W\in\Whit^{\fpsi_N}(\Pi)$ and $\typZe^\vee\in\Whit^{\fpsi_N^{-1}}(\sigma^\vee)_\flat$ we have
\[
\int_{\bar U_\mir} \Bil_0(MW,\typZe^\vee(\cdot\bar v)) \ d\bar v= \int_{N_M\bs\mir_M} MW(\ww l)
\typZe^\vee(\ww l) \delta_{\bar U_\mir}^{-1}(l)\ dl.
\]
\end{lemma}

\begin{proof}
Note that $P\bar U\cap\mir=U_\mir\mir_M\bar U_\mir$.
Thus, the left-hand side is
\begin{multline}\label{eq: A5}
\int_{\bar U_\mir} \int_{N_\mir\bs\mir}MW(p )\typZe^\vee( p\bar v) \ dp \ d\bar v
\\=\int_{\bar U_\mir}\int_{N_M\bs\mir_M}\int_{\bar U_\mir} MW(\bar u l )\typZe^\vee( \bar u l\bar v)\ d\bar u \ dl\ d\bar v
\\=\int_{N_M\bs\mir_M}\int_{\bar U_\mir}\int_{\bar U_\mir} MW(\bar u l )\typZe^\vee( \bar u l\bar v)\ d\bar v\ d\bar u \ dl
\\=\int_{N_M\bs\mir_M}\int_{\bar U_\mir}\int_{\bar U_\mir} MW(\bar u l)\typZe^\vee( \bar v l)
\delta_{\bar U_\mir}^{-1}(l)\ d\bar v \ d\bar u\ dl
\end{multline}
where we made a change of variable $\bar v\mapsto l^{-1}\bar u^{-1}\bar v l$.
By the condition on $\typZe^\vee\rest_\mir$ and Lemma~\ref{L: support}, the integrand is compactly supported,
which justifies the previous steps.
Applying Lemma~\ref{L: A1} for both integrals over $\bar U_\mir$ we get the required statement.
\end{proof}

Since \eqref{eq: innercomp} holds up to a scalar, in order to conclude Proposition \ref{prop: appmain},
it suffices, in view of Lemmas \ref{L: A3} and \ref{L: A4}, to show the existence of
$W\in \Whit^{\fpsi_N}(\Pi)_\sharp$ and $\typZe^\vee\in\Whit^{\fpsi_N^{-1}}(\sigma^\vee)_\flat$ such that
the right-hand side of the expression in Lemma~\ref{L: A3}, or equivalently, the expression \eqref{eq: A5}, is nonzero.
By Corollary~\ref{C: restomir2}, given $\phi\in\swrz(\bar U_\mir)$ and $W'\in\ind_{N_M}^{\mir_M}\fpsi_{N_M}^{-1}$
there exists (a unique) $\typZe^\vee\in\Whit^{\fpsi_N^{-1}}(\sigma^\vee)_\flat$ such that
\[
\typZe^\vee(ul\bar v)=\phi(\bar v)W'(l)\ \ \forall u\in U_\mir,\ l\in\mir_M,\ \bar v\in\bar U_\mir.
\]
Thus, we only need to show that $\int_{\bar U_\mir} MW(\bar u)\ d\bar u$, or equivalently by Lemma~\ref{L: A1},
$MW(\ww)$, is nonzero for some $W\in \Whit^{\fpsi_N}(\Pi)_\sharp$.
However, this is clear since $MW(\ww)=\int_{\bar U}W(\bar u)\ d\bar u$.

This finishes the proof of Proposition \ref{prop: appmain}.

\def\cprime{$'$}

\end{document}